\documentclass[10pt,letterpaper]{amsart}
\usepackage{multicol,amssymb,amsmath,amsbsy,amsthm,graphicx,esvect,comment,caption,epsfig,subfigure,epstopdf}

\usepackage{cleveref}
\usepackage[latin1]{inputenc}

\theoremstyle{plain}
\newtheorem{theorem}{Theorem}[section]
\newtheorem{lemma}[theorem]{Lemma}

\newtheorem{proposition}[theorem]{Proposition}
\newtheorem{conjecture}[theorem]{Conjecture}

\theoremstyle{remark}
\newtheorem{remark}[theorem]{Remark}

\theoremstyle{definition}
\newtheorem{definition}[theorem]{Definition}

\newcommand{\card}{\text{\rm{card}}}

\begin{document}

\title[Asymptotics of the energy of sections of greedy energy sequences]
{Asymptotics of the energy of sections of greedy energy sequences on the unit circle, and some conjectures for general sequences}

\author{Abey L\'{o}pez-Garc\'{i}a}
\address{University of South Alabama, Department of Mathematics and Statistics, ILB 325, 411 University Blvd North, Mobile AL, 36688.}
\email{lopezgarcia@southalabama.edu}

\author{Douglas A. Wagner}
\address{University of South Alabama, Department of Mathematics and Statistics, ILB 325, 411 University Blvd North, Mobile AL, 36688.}
\email{daw802@jagmail.southalabama.edu}
\thanks{The results of this paper form a part of the second author's M. Sc. dissertation at the University of South Alabama. The first author was partially supported by the grant MTM2012-36732-C03-01 of the Spanish Ministry of Economy and Competitiveness.}

\subjclass[2010]{31C20, 52A40 (primary), and 31C15 (secondary)}

\keywords{Greedy energy sequence; Leja sequence; Riesz kernel; equilibrium measure; optimal energy configuration.}

\begin{abstract}
In this paper we investigate the asymptotic behavior of the Riesz $s$-energy of the first $N$ points of a greedy $s$-energy sequence on the unit circle, for all values of $s$ in the range $0\leq s<\infty$ (identifying as usual the case $s=0$ with the logarithmic energy). In the context of the unit circle, greedy $s$-energy sequences coincide with the classical Leja sequences constructed using the logarithmic potential. We obtain first-order and second-order asymptotic results. The key idea is to express the Riesz $s$-energy of the first $N$ points of a greedy $s$-energy sequence in terms of the binary representation of $N$. Motivated by our results, we pose some conjectures for general sequences on the unit circle.
\end{abstract}

\dedicatory{Dedicated to Ed Saff on the occasion of his $70$th birthday}

\maketitle

\section{Introduction and statement of main results}\label{section:intro}

This work can be regarded as a continuation of some of the investigations initiated by the first author and Saff in \cite{LopSaff}, where greedy energy sequences with respect to general kernels were introduced and many properties of these sequences were obtained. Greedy energy sequences are defined in \cite{LopSaff} as follows. Let $X$ be a locally compact space containing infinitely many points, and let $k: X\times X\longrightarrow \mathbb{R}\cup\{+\infty\}$ be a symmetric and lower semicontinuous function. The function $k$ is referred to as the \emph{kernel}. Given a compact set $A\subset X$, a sequence $(a_{n})_{n=0}^{\infty}\subset A$ is called a \emph{greedy $k$-energy sequence} on $A$ if it is generated in the following way:
\begin{itemize}
\item[$\bullet$] The first point $a_{0}$ is selected arbitrarily on $A$.
\item[$\bullet$] Assuming that $a_{0},\ldots,a_{n}$ have been selected, $a_{n+1}$ is chosen to satisfy
\begin{equation}\label{eq:alggreedy}
\sum_{i=0}^{n} k(a_{n+1},a_{i})=\inf_{x\in A} \sum_{i=0}^{n} k(x,a_{i})
\end{equation}
for every $n\geq 0$.
\end{itemize}

The existence of a point $a_{n+1}$ satisfying \eqref{eq:alggreedy} follows from the lower semicontinuity of $k$, but of course the choice of $a_{n+1}$ may not be unique in general. As in \cite{LopSaff}, we will use here the notation $\alpha_{N}:=(a_{0},\ldots,a_{N-1})$ to denote the first $N$ points of the sequence $(a_{n})_{n=0}^{\infty}$.

If $X=\mathbb{C}$ and we use the logarithmic kernel $k(x,y)=-\log |x-y|$ (here and below $|x-y|$ indicates the Euclidean distance between $x$ and $y$), then the above algorithm generates the classical Leja sequences on compact subsets of the complex plane. If $X=\mathbb{R}^{p}$, $p\geq 2$, and $k(x,y)=|x-y|^{-s}$ is the Riesz $s$-kernel with parameter $s>0$, then the sequences obtained are the greedy $s$-energy sequences on compact subsets of $\mathbb{R}^{p}$.

The main interest for the study of greedy energy sequences in \cite{LopSaff} was to compare the asymptotic behavior of the configurations $\alpha_{N}$ with the asymptotic behavior of optimal energy configurations. The study of the asymptotic properties of optimal energy configurations has been a leitmotif in the work of Saff, and his efforts have produced a tremendous advancement in the theory of discrete minimal energy problems. In this work we will be referring frequently to optimal energy configurations and their asymptotic properties, so we define them next.

Given a set $\omega=\{x_{1},\ldots,x_{N}\}$ of $N$ ($N\geq 2$) points in $X$, not necessarily distinct, we write $\card(\omega)=N$ and we define the \emph{discrete energy} of $\omega$ with respect to $k$ by
\begin{equation}\label{def:energy}
E(\omega):=\sum_{1\leq i\neq j\leq N}k(x_{i},x_{j})=2\sum_{1\leq i<j\leq N} k(x_{i},x_{j}).
\end{equation}
Now assume that $A\subset X$ is a compact set and $N\geq 2$ is an integer. A set $\omega_{N}\subset A$ is an \emph{optimal $N$-point configuration} on $A$ if
\[
E(\omega_{N})=\inf\{E(\omega): \omega\subset A, \card(\omega)=N\},
\]
that is, $\omega_{N}$ has the lowest possible energy among all $N$-point configurations on $A$. Note that the existence of optimal energy configurations is guaranteed by the lower semicontinuity of $k$.

Part of the results in \cite{LopSaff} were obtained in the context of the unit circle $S^{1}$ and the Riesz $s$-kernel. It was shown in \cite{LopSaff} that in terms of first-order asymptotics, there is a difference in the behavior of greedy $s$-energy sequences and optimal $N$-point configurations when $s>1$ (see \cite[Proposition 2.6]{LopSaff}) which is not present in the case $s\leq 1$. Moreover, this difference takes place in the more general context of rectifiable Jordan arcs or curves, see \cite[Theorem 2.9]{LopSaff}. It was also shown in \cite{LopSaff} that on $S^{1}$ and in the case $0<s\leq 1$, the second-order asymptotic behavior of greedy $s$-energy sequences is no longer the same as that of optimal $N$-point configurations (see \cite[Propositions 2.4 and 2.7]{LopSaff}). These differences will be explained in detail below.

In this paper we investigate more deeply the asymptotic behavior of Leja sequences and greedy $s$-energy sequences on $S^{1}$ from the energy point of view. Consequently we are able to refine some of the results in \cite{LopSaff} mentioned above. We first describe the results we obtain for Leja sequences and later for greedy $s$-energy sequences. The results we obtain have also motivated some conjectures for general sequences on $S^{1}$ that we state in Section~\ref{section:conjectures}. 

\subsection{Results for Leja sequences on the unit circle}

Recall that a Leja sequence $(a_{n})_{n=0}^{\infty}$ on an infinite compact set $K\subset\mathbb{C}$ is a sequence that is constructed by choosing an arbitrary $a_{0}\in K$, and selecting each subsequent $a_{n+1}\in K$ such that  
\begin{equation}\label{Lejaproperty}
\sum_{i=0}^{n}\log\frac{1}{|a_{n+1}-a_{i}|}=\inf_{z\in K}\sum_{i=0}^{n}\log\frac{1}{|z-a_{i}|},\qquad n\geq 0.
\end{equation}
Equivalently, for every $n\geq 0$, $a_{n+1}$ maximizes the product $\prod_{i=0}^{n}|z-a_{i}|$ on $K$. Leja sequences are named after F. Leja in recognition of his work \cite{Leja}, although they were first introduced by A. Edrei in \cite{Edrei}. These sequences have attracted some interest in recent years, especially concerning the study of their interpolation properties, see e.g. \cite{BiaCal, CalManh, Chkifa, Reichel, TaylorTotik}. Not many works have been devoted to the study of the energy and distribution of Leja sequences; some of these are \cite{Gotz, Pritsker, CD, Lop, LopSaff}.

Given a configuration $\omega=\{x_{1},\ldots,x_{N}\}$ of $N\geq 2$ distinct points in the complex plane $\mathbb{C}$, we will denote by $E_{0}(\omega)$ its logarithmic energy, that is,
\[
E_{0}(\omega):=\sum_{1\leq i\neq j\leq N}\log\frac{1}{|x_{i}-x_{j}|}=2\sum_{1\leq i<j\leq N}\log\frac{1}{|x_{i}-x_{j}|},
\]
see \eqref{def:energy}.

Let $(a_{n})_{n=0}^{\infty}$ be a Leja sequence on a compact set $K\subset\mathbb{C}$, and recall that $\alpha_{N}=(a_{0},\ldots,a_{N-1})$ denotes the $N$-tuple of the first $N$ points of this sequence. A well-known result in logarithmic potential theory that can be consulted in \cite[Theorem V.1.1]{SaffTotik} asserts that if $K$ is non-polar (i.e., $K$ supports a positive measure with finite logarithmic energy), then
\begin{equation}\label{eq:firstorderasymp}
\lim_{N\rightarrow\infty}\frac{E_{0}(\alpha_{N})}{N^{2}}=\inf_{\mu\in\mathcal{P}(K)}\iint\log\frac{1}{|z-w|}\,d\mu(w)\,d\mu(z),
\end{equation}
where $\mathcal{P}(K)$ is the set of probability measures supported on $K$. Moreover, the sequence of point configurations $\alpha_{N}$ has as limiting distribution the equilibrium measure on $K$, which is the unique probability measure $\nu$ on $K$ satisfying the extremal property
\begin{equation}\label{contlogenergy}
\iint\log\frac{1}{|z-w|}\,d\nu(w)\,d\nu(z)=\inf_{\mu\in \mathcal{P}(K)}\iint\log\frac{1}{|z-w|}\,d\mu(w)\,d\mu(z).
\end{equation}
The asymptotic result \eqref{eq:firstorderasymp} was first proved for Fekete sets on $K$ by Fekete and Szeg\H{o}, see \cite[Theorem 5.5.2]{Ran}. Fekete sets on $K$ consisting of $N\geq 2$ points are exactly optimal $N$-point configurations on $K$ relative to the logarithmic kernel; that is, they are configurations $\omega_{N}\subset K$ satisfying
\[
E_{0}(\omega_{N})=\inf\{E_{0}(\omega): \omega\subset K, \card(\omega)=N\}.
\]

Before we state our results for Leja sequences on $S^{1}$ we describe some basic properties of these sequences. Firstly, it is clear that a rotation (by multiplication with $\rho\in\mathbb{C}$, $|\rho|=1$) will neither destroy the Leja sequence property \eqref{Lejaproperty} nor change the logarithmic energy of the configurations $\alpha_{N}$. So, it suffices to consider Leja sequences starting with initial point $1$. Following the terminology used in \cite{BiaCal} and \cite{CalManh}, we will refer to the configurations $\alpha_{N}$ as \emph{$N$-Leja sections}. 

Leja sequences on $S^{1}$ can be described in detail by the following properties obtained by L. Bialas-Ciez and J.-P. Calvi in \cite[Theorem 5]{BiaCal}, see also \cite[Lemma 4.2]{LopSaff}. Let us define first the notation $(A,B)=(a_{0},\ldots,a_{N-1},b_{0},\ldots,b_{M-1})$ for an $N$-tuple $A=(a_{0},\ldots,a_{N-1})$ and an $M$-tuple $B=(b_{0},\ldots,b_{M-1})$. 
Then:
\begin{itemize}
\item[$1)$] Any $2^{n}$-Leja section is formed by the $2^{n}$th roots of unity. 

\item[$2)$] Given any $2^{n+1}$-Leja section $\alpha_{2^{n+1}}$ containing the $2^{n}$-Leja section $\alpha_{2^{n}}$ as its first $2^{n}$ points, there exists a $2^{n}$th root $\rho$ of $-1$ and a $2^{n}$-Leja section $\beta_{2^{n}}$ such that $\alpha_{2^{n+1}}=(\alpha_{2^{n}},\rho\,\beta_{2^{n}})$. 

\item[$3)$] Iterating $2)$, it is easily seen that for any $k$-Leja section $\alpha_{k}$ with $k=2^{n_{1}}+2^{n_{2}}+\cdots+2^{n_{t}}$, $n_{1}>n_{2}>\cdots >n_{t}\geq0$, there exists for each $i=1,\ldots,t$ a $2^{n_{i}}$-Leja section $\alpha^{i}_{2^{n_{i}}}$ (with initial point $1$) such that 
\[
\alpha_{k}= (\alpha^{1}_{2^{n_{1}}},\rho_{1}\,\alpha^{2}_{2^{n_{2}}},\rho_{1}\rho_{2}\,\alpha^{3}_{2^{n_{3}}},\ldots, \left(\prod_{i=1}^{t-1}\rho_{i}\right)\alpha^{t}_{2^{n_{t}}}),
\] 
for some numbers $\rho_{i}$ that are $2^{n_{i}}$th roots of $-1$. In other words, any Leja section is composed of rotated Leja sections of smaller size. 
\end{itemize}

Concerning the asymptotic behavior of $E_{0}(\alpha_{N})$ for Leja sequences on the unit circle, the asymptotic formula \eqref{eq:firstorderasymp} applied in this context gives
\[
\lim_{N\rightarrow\infty}\frac{E_{0}(\alpha_{N})}{N^2}=0,
\]
since the equilibrium measure on $S^{1}$ is the normalized arclength measure and its logarithmic energy \eqref{contlogenergy} is zero. In this paper we prove the following.

\begin{theorem}\label{theo:firstorderasymp}
If $(a_{n})_{n=0}^{\infty}$ is a Leja sequence on $S^{1}$, then for the sequence $\alpha_{N}=(a_{0},\ldots,a_{N-1})$ we have 
\begin{equation}\label{firstorderlimit}
\lim_{N\rightarrow\infty}\frac{E_{0}(\alpha_{N})}{N\log N}=-1.
\end{equation}
\end{theorem}

We note that each Fekete set on $S^{1}$ with $N\geq 2$ points is a rotated copy of the $N$th roots of unity having logarithmic energy $-N\log N$. Therefore $E_{0}(\omega)\geq -N\log N$ for any $N$-point configuration $\omega$ on $S^{1}$. A refinement of \eqref{firstorderlimit} is the following second order estimate.

\begin{theorem}\label{theo:secondorderest}
Under the same assumptions as in Theorem $\ref{theo:firstorderasymp}$, for every $N$ we have
\begin{equation}\label{secondorderest}
0\leq \frac{E_{0}(\alpha_{N})+N\log(N)}{N}<\log(4/3).
\end{equation}
The upper bound in \eqref{secondorderest} is best possible since
\begin{equation}\label{ulsecondorder}
\limsup_{N\rightarrow\infty}\frac{E_{0}(\alpha_{N})+N\log N}{N}=\log(4/3).
\end{equation}
\end{theorem}

Observe that if $N$ is a power of $2$, then the lower bound in \eqref{secondorderest} is attained. The estimates in \eqref{secondorderest} imply \eqref{firstorderlimit}, but we shall provide a direct proof of \eqref{firstorderlimit} not using \eqref{secondorderest}.

\subsection{Results for greedy $s$-energy sequences on the unit circle}

Let $s>0$ and let $\omega=\{x_{1},\ldots,x_{N}\}\subset\mathbb{C}$ be a configuration of $N\geq 2$ distinct points. We denote by $E_{s}(\omega)$ the Riesz $s$-energy of $\omega$, that is,
\[
E_{s}(\omega):= \sum_{1\leq i\neq j\leq N}\frac{1}{|x_{i}-x_{j}|^{s}}= 2\sum_{1\leq i< j\leq N}\frac{1}{|x_{i}-x_{j}|^{s}}.
\]

In this paper we shall also analyze the asymptotic behavior of the Riesz $s$-energy of the first $N$ points of a greedy $s$-energy sequence on $S^{1}$. Recall that by definition, such sequences $(a_{n})_{n=0}^{\infty}\subset S^{1}$ are obtained by choosing an arbitrary $a_{0}\in S^{1}$ and selecting each subsequent $a_{n+1}\in S^{1}$, $n\geq 0$, such that
\[
\sum_{i=0}^{n}\frac{1}{|a_{n+1}-a_{i}|^{s}}=\inf_{z\in S^{1}}\sum_{i=0}^{n}\frac{1}{|z-a_{i}|^{s}}.
\] 
The first important observation we make is that for any $s>0$, greedy $s$-energy sequences coincide with Leja sequences on $S^{1}$ due to the symmetry of the circle. This can be easily deduced from an induction argument that uses Lemmas 4.1 and 4.2 from \cite{LopSaff} which we will omit here. 

Also, we emphasize that for any fixed $s>0$, if $\omega_{N}$ is an optimal $N$-point configuration on $S^{1}$ that minimizes the Riesz $s$-energy, i.e., $\omega_{N}$ satisfies
\[
E_{s}(\omega_{N})=\inf\{E_{s}(\omega): \omega\subset S^{1}, \card(\omega)=N\},
\]
then $\omega_{N}$ is again formed by $N$ equally spaced points, see \cite{Gotz2}. 

Following the notation used in \cite{BHS}, we will denote by $\mathcal{L}_{s}(N)$ the Riesz $s$-energy ($s>0$) of $N$ equally spaced points on the unit circle, i.e.,
\[
\mathcal{L}_{s}(N):=E_{s}(\{z_{k,N}\}_{k=1}^{N}),\qquad z_{k,N}:=\exp(2\pi i (k-1)/N),\quad k=1,\ldots,N.
\] 
It is easy to see that
\[
\mathcal{L}_{s}(N)=2^{-s} N \sum_{k=1}^{N-1}\left(\sin\frac{k\pi}{N}\right)^{-s},\qquad N\geq 2,
\]
using $|e^{i\xi}-e^{i\theta}|=2 |\sin (\frac{\xi-\theta}{2})|$. By convention we set $\mathcal{L}_{s}(1)=0$.

If the Riesz parameter $s$ satisfies $0<s<1$, one can still use potential theory, as in the logarithmic case, to study the asymptotic behavior of $\mathcal{L}_{s}(N)$ and the Riesz $s$-energy of greedy $s$-energy configurations. The following first-order asymptotic results are known and can be proved using the same techniques. We have
\begin{equation}\label{firstorderRiesz}
\lim_{N\rightarrow\infty}\frac{\mathcal{L}_{s}(N)}{N^{2}}=\lim_{N\rightarrow\infty}\frac{E_{s}(\alpha_{N})}{N^{2}}=I_{s}(\sigma),
\end{equation}
where $\sigma$ is the normalized arc length measure on $S^{1}$, which minimizes the energy
\begin{equation}\label{def:continuousRieszenergy}
I_{s}(\mu):=\iint\frac{1}{|x-y|^{s}}\,d\mu(x)\,d\mu(y)
\end{equation}
among all probability measures on $S^{1}$. For a proof of \eqref{firstorderRiesz}, see \cite{Landkof,LopSaff}. The limiting value in \eqref{firstorderRiesz} is given by
\[
I_{s}(\sigma)=\frac{1}{2\pi}\int_{0}^{2\pi}\frac{1}{|1-e^{i\phi}|^s}\,d\phi=2^{-s}\,\frac{\Gamma((1-s)/2)}{\sqrt{\pi}\,\Gamma(1-s/2)},
\]
cf. \cite[2.5.3.1]{Prud}.

In terms of second-order asymptotics for the sequence $\mathcal{L}_{s}(N)$, the following limit holds (see \cite{BHS}):
\begin{equation}\label{eq:secorderasympRieszFekete}
\lim_{N\rightarrow\infty}\frac{\mathcal{L}_{s}(N)-I_{s}(\sigma) N^{2}}{N^{1+s}}=\frac{2\zeta(s)}{(2\pi)^{s}},
\end{equation}
where $\zeta(s)$ is the analytic extension of the classical Riemann zeta function. It should be noted that in the range $s\in(0,1)$ we have $\zeta(s)<0$. In contrast to \eqref{eq:secorderasympRieszFekete}, it was shown in \cite[Corollary 2.5]{LopSaff} that in the case of greedy $s$-energy sequences on $S^{1}$ and the corresponding configurations $\alpha_{N}$, the sequence
\[
\left(\frac{E_{s}(\alpha_{N})-I_{s}(\sigma) N^{2}}{N^{1+s}}\right)_{N}
\]
is not convergent. In this paper we look more closely at this sequence.

In order to state our results in the Riesz setting, we need to introduce certain notations and definitions.

\begin{definition}\label{def:Thetap}
Let $p\geq 1$ be a fixed integer. We let $\Theta_{p}\subset [0,1]^{p}$ denote the set of all vectors $\vec{\theta}=(\theta_{1},\theta_{2},\ldots,\theta_{p})$ for which there exists an infinite sequence $\mathcal{N}$ of integers $N=2^{n_{1}}+2^{n_{2}}+\cdots+2^{n_{p}}$, $n_{1}>n_{2}>\cdots>n_{p}\geq 0$, satisfying
\begin{equation}\label{eq:conddefthetavec}
\lim_{N\in\mathcal{N}}\frac{2^{n_{i}}}{N}=\theta_{i},\qquad \mbox{for all}\,\,i=1,\ldots,p.
\end{equation}  
\end{definition}

\smallskip

Note that if $(\theta_{1},\ldots,\theta_{p})\in\Theta_{p}$, then 
\begin{equation}\label{eq:propvecinTheta}
\sum_{i=1}^{p}\theta_{i}=1.
\end{equation}
On $\Theta_{p}\times [0,\infty)$ we define the following function
\begin{equation}\label{def:H}
H((\theta_{1},\ldots,\theta_{p});s):=\sum_{k=1}^{p}\theta_{k}^{s}\,\Big(2 (2^s-1) \left(\sum_{j=k+1}^{p}\theta_{j}\right)+\theta_{k}\Big).
\end{equation}
It follows from \eqref{eq:propvecinTheta} that for any $\vec{\theta}=(\theta_{1},\ldots,\theta_{p})\in\Theta_{p}$ we have
\[
H(\vec{\theta};0)=H(\vec{\theta};1)=1.
\]
In Section~\ref{section:RieszenergyLeja} we give some further remarks about the sets $\Theta_{p}$ and the functions $H$ in \eqref{def:H}.
The graphs of some functions $H$ associated with three vectors $\vec{\theta}$ are shown in Figure~\ref{graphsHfunctions}.

\begin{figure}[h]
\centering
\includegraphics[totalheight=2.5in,keepaspectratio]{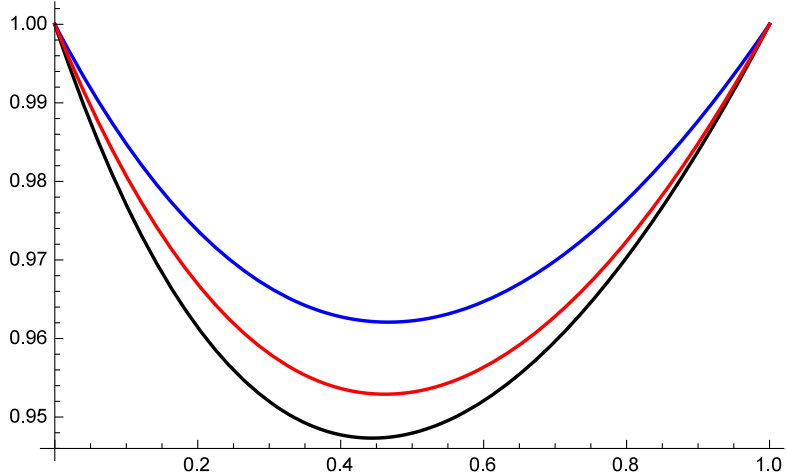}
\caption{In increasing order, we show the graphs of the functions \eqref{def:H} associated with the vectors $\vec{\theta}=(16/21, 4/21, 1/21)$, $\vec{\theta}=(4/5,1/5)$ and $\vec{\theta}=(2/3,1/3)$, respectively.}\label{graphsHfunctions}
\end{figure}

\begin{definition}\label{def:functionsh}
Let $0<s<1$ be fixed. Using the function \eqref{def:H} we introduce the notations
\begin{align}
\underline{h}_{p}(s) & :=\inf_{\vec{\theta}\in\Theta_{p}} H(\vec{\theta};s),\qquad p\in\mathbb{N},\label{def:hplower:1}\\
\underline{h}(s) & :=\inf_{p\in\mathbb{N}} \underline{h}_{p}(s).\label{def:hplower:2}
\end{align}
Similarly, for $s>1$ fixed we define
\begin{align}
\overline{h}_{p}(s) & :=\sup_{\vec{\theta}\in\Theta_{p}} H(\vec{\theta};s),\qquad p\in\mathbb{N},\label{def:hpupper:1}\\
\overline{h}(s) & :=\sup_{p\in\mathbb{N}} \overline{h}_{p}(s).\label{def:hpupper:2}
\end{align}
\end{definition}

Unfortunately we have not found an explicit expression of the functions $\underline{h}(s)$ and $\overline{h}(s)$. Our next result is the following.

\begin{theorem}\label{theo:secorderasympRieszenergy:1}
Let $s\in(0,1)$ be fixed, and let $(a_{n})_{n=0}^{\infty}$ be a greedy $s$-energy sequence on $S^{1}$. Then, for the sequence of configurations $\alpha_{N}=(a_{n})_{n=0}^{N-1}$ we have
\begin{equation}\label{eq:secorderasympRieszenergy:1}
\limsup_{N\rightarrow\infty}\frac{E_{s}(\alpha_{N})-I_{s}(\sigma) N^{2}}{N^{1+s}}=\underline{h}(s)\,\frac{2\zeta(s)}{(2\pi)^{s}},
\end{equation}
where $\zeta(s)$ is the analytic extension of the classical Riemann zeta function, and $\underline{h}(s)$ is defined in \eqref{def:hplower:1}--\eqref{def:hplower:2}. We also have
\begin{equation}\label{secondorderasympRieszenergy:2}
\liminf_{N\rightarrow\infty}\frac{E_{s}(\alpha_{N})-I_{s}(\sigma) N^{2}}{N^{1+s}}=\frac{2\zeta(s)}{(2\pi)^{s}}.
\end{equation}
In particular, the sequence $\frac{E_{s}(\alpha_{N})-I_{s}(\sigma) N^{2}}{N^{1+s}}$ is not convergent since $\underline{h}(s)<1$ for every $s\in (0,1)$.
\end{theorem}

In contrast to the case $s\in(0,1)$, if $s\geq 1$ potential-theoretic tools are no longer available to study the asymptotic behavior of $\mathcal{L}_{s}(N)$ or $E_{s}(\alpha_{N})$. This is due to the fact that in this case the continuous Riesz $s$-energy \eqref{def:continuousRieszenergy} of any probability measure $\mu$ on $S^{1}$ is infinite. 

As a particular case of a general result for rectifiable Jordan curves in $\mathbb{R}^{d}$ proved in \cite[Theorem 3.2]{MMRS}, we know that if $s>1$ then
\begin{equation}\label{eq:asympfirstorderRieszsupercrit}
\lim_{N\rightarrow\infty}\frac{\mathcal{L}_{s}(N)}{N^{1+s}}=\frac{2\zeta(s)}{(2\pi)^{s}},
\end{equation}
where $\zeta(s)=\sum_{n=1}^{\infty}n^{-s}$ denotes now the classical Riemann zeta function. 

Concerning greedy $s$-energy sequences, we have the following result, analogous to Theorem~\ref{theo:secorderasympRieszenergy:1}.

\begin{theorem}\label{theo:firstorderasympRieszenergy}
Let $s>1$ be fixed, and let $(a_{n})_{n=0}^{\infty}$ be a greedy $s$-sequence on $S^{1}$. Then, for the sequence of configurations $\alpha_{N}=(a_{n})_{n=0}^{N-1}$ we have
\begin{equation}\label{eq:firstorderasympRieszenergy:1}
\limsup_{N\rightarrow\infty}\frac{E_{s}(\alpha_{N})}{N^{1+s}}=\overline{h}(s)\,\frac{2\zeta(s)}{(2\pi)^{s}},
\end{equation}
where $\zeta(s)$ is the classical Riemann zeta function, and $\overline{h}(s)$ is defined in \eqref{def:hpupper:1}--\eqref{def:hpupper:2}. We also have
\begin{equation}\label{eq:firstorderasympRieszenergy:2}
\liminf_{N\rightarrow\infty}\frac{E_{s}(\alpha_{N})}{N^{1+s}}=\frac{2\zeta(s)}{(2\pi)^{s}}.
\end{equation}
In particular, the sequence $\frac{E_{s}(\alpha_{N})}{N^{1+s}}$ is not convergent since $\overline{h}(s)>1$ for every $s>1$.
\end{theorem}

We remark that in \cite[Proposition 2.6]{LopSaff} it was already shown that the sequence $\frac{E_{s}(\alpha_{N})}{N^{1+s}}$ is not convergent. We also want to emphasize that the following result, related with \eqref{eq:firstorderasympRieszenergy:1}, can be deduced from \cite[Theorem 2.9]{LopSaff}. If $(x_{n})_{n=0}^{\infty}\subset S^{1}$ is any sequence of pairwise distinct points on the unit circle and $s>1$, then for the sequence of configurations $\omega_{N}=\{x_{0},\ldots,x_{N-1}\}$ we have
\[
\limsup_{N\rightarrow\infty}\frac{E_{s}(\omega_{N})}{N^{1+s}}>\frac{2\zeta(s)}{(2\pi)^{s}}.
\]

We finally consider the critical case $s=1$. As a corollary of \cite[Theorem 3.2]{MMRS} and \cite[Theorem 2.10]{LopSaff} we know that
\[
\lim_{N\rightarrow\infty}\frac{\mathcal{L}_{1}(N)}{N^2 \log N}=\lim_{N\rightarrow\infty}\frac{E_{1}(\alpha_{N})}{N^2 \log N}=\frac{1}{\pi},
\]
for any greedy $s$-energy sequence ($s=1$) on $S^{1}$ and the corresponding configurations $\alpha_{N}$. Moreover, we have the following second-order asymptotics (see \cite{BHS}):
\begin{equation}\label{eq:secorderFeketecrit}
\lim_{N\rightarrow\infty}\frac{\mathcal{L}_{1}(N)-\frac{1}{\pi}N^2 \log N}{N^2}=\frac{1}{\pi}(\gamma-\log(\pi/2)),
\end{equation}
where $\gamma=\lim_{N\rightarrow\infty} (1+\frac{1}{2}+\cdots+\frac{1}{N}-\log N)$ denotes the Euler-Mascheroni constant.

In our next result we consider the corresponding second-order expression
\[
\frac{E_{1}(\alpha_{N})-\frac{1}{\pi}\,N^{2}\,\log N}{N^2}.
\]
In order to state this result we need some definitions. 
For $\vec{\theta}=(\theta_{1},\ldots,\theta)\in\Theta_{p}$ we let
\begin{equation}\label{def:K}
K(\theta_{1},\ldots,\theta_{p}):= 2\log 2+\sum_{k=1}^{p}\theta_{k}^2\, \log \left(\theta_{k}/4\right)+2\sum_{k=1}^{p-1}\left(\sum_{j=k+1}^{p}\theta_{j}\right) \theta_{k}\log \theta_{k},
\end{equation}
where if $\theta_{k}=0$, we understand in \eqref{def:K} that $\theta_{k} \log \theta_{k}=0$. Let
\begin{equation}\label{def:kappa}
\kappa:=\sup_{p\in\mathbb{N}}\, \sup_{\vec{\theta}\in\Theta_{p}} K(\vec{\theta}).
\end{equation}

\begin{theorem}\label{theo:criticalcase}
Let $(a_{n})_{n=0}^{\infty}$ be a greedy $s$-energy sequence on $S^{1}$ for $s=1$. Then, for the sequence of configurations $\alpha_{N}=(a_{n})_{n=0}^{N-1}$ we have
\begin{equation}\label{secondorderasympcritcase:1}
\limsup_{N\rightarrow\infty}\frac{E_{1}(\alpha_{N})-\frac{1}{\pi}N^2 \log N}{N^2}=\frac{1}{\pi}(\gamma-\log(\pi/2)+\kappa),
\end{equation}
where $\kappa$ is the constant in \eqref{def:kappa}. We also have
\begin{equation}\label{secondorderasympcritcase:2}
\liminf_{N\rightarrow\infty}\frac{E_{1}(\alpha_{N})-\frac{1}{\pi}N^2 \log N}{N^2}=\frac{1}{\pi}(\gamma-\log(\pi/2)).
\end{equation}
In particular, the sequence $\frac{E_{1}(\alpha_{N})-\frac{1}{\pi}N^2 \log N}{N^2}$ is not convergent since $\kappa>0$.
\end{theorem}

We remark that in \cite[Corollary 2.8]{LopSaff} it was already shown that the sequence $\frac{E_{1}(\alpha_{N})-\frac{1}{\pi} N^{2}\log N}{N^2}$ is not convergent.

This paper is organized as follows. In Section~\ref{section:conjectures} we formulate some conjectures for general sequences on the unit circle. In Section~\ref{section:firstorderasymp} we prove Theorem~\ref{theo:firstorderasymp}, and in Section~\ref{section:secondorderest} we prove Theorem~\ref{theo:secondorderest}. In Section~\ref{section:RieszenergyLeja} we give the proofs of the results in the Riesz setting.

\section{Some conjectures}\label{section:conjectures}

From the energy point of view, it is clear that Leja sequences and greedy $s$-energy sequences are special within the class of general sequences on the unit circle, as each point in the sequence is selected in an optimal way. In fact, we can also define the point $a_{n}$ in a greedy $s$-energy sequence as a point satisfying
\[
E_{s}(\{a_{0},\ldots,a_{n-1},a_{n}\})=\inf_{x\in S^{1}} E_{s}(\{a_{0},\ldots,a_{n-1},x\}),\qquad n\geq 1.
\]
Because of this property, it is reasonable to expect that greedy sequences provide the lowest upper limit for the normalized energy expressions that have been described above. We state this as a conjecture.

\begin{conjecture}
Let $(x_{n})_{n=0}^{\infty}\subset S^{1}$ be an arbitrary sequence on $S^{1}$ such that $x_{i}\neq x_{j}$ for every $i\neq j$, and let $\omega_{N}=\{x_{0},\ldots,x_{N-1}\}$, $N\geq 2$. Then
\[
\limsup_{N\rightarrow\infty}\frac{E_{0}(\omega_{N})+N\log N}{N}\geq \log(4/3);
\]
for $s\in(0,1)$,
\[
\limsup_{N\rightarrow\infty}\frac{E_{s}(\omega_{N})-I_{s}(\sigma) N^{2}}{N^{1+s}}\geq \underline{h}(s)\,\frac{2\zeta(s)}{(2\pi)^{s}};
\]
for $s>1$,
\[
\limsup_{N\rightarrow\infty}\frac{E_{s}(\omega_{N})}{N^{1+s}}\geq \overline{h}(s)\,\frac{2\zeta(s)}{(2\pi)^{s}};
\]
and for $s=1$,
\[
\limsup_{N\rightarrow\infty}\frac{E_{s}(\omega_{N})-\frac{1}{\pi}N^{2}\log N}{N^2}\geq \frac{1}{\pi}(\gamma-\log(\pi/2)+\kappa),
\]
where the expressions on the right-hand sides of the last three inequalities are the same as those appearing 
in \eqref{eq:secorderasympRieszenergy:1}, \eqref{eq:firstorderasympRieszenergy:1} and \eqref{secondorderasympcritcase:1}. 
\end{conjecture}

\section{First-order asymptotics in the logarithmic case}\label{section:firstorderasymp}

In this section we prove Theorem~\ref{theo:firstorderasymp}, but first we give a preliminary discussion and prove an auxiliary result.

Let $(a_{n})_{n=0}^{\infty}$ be a Leja sequence on $S^{1}$. It was shown in \cite[Lemma 4]{CalManh} that if $k=2^{n_{1}}+2^{n_{2}}+\cdots+2^{n_{t}}$ with $n_{1}>n_{2}>\cdots> n_{t}\geq 0$, then
\begin{equation}\label{eq:minprod}
\prod_{i=0}^{k-1}|a_{k}-a_{i}|=2^{t}.
\end{equation}
Let $U_{k}$ denote the discrete potential
\[
U_{k}(z):=\sum_{i=0}^{k-1}\log\frac{1}{|z-a_{i}|}.
\]
Then
\[
E_{0}(\alpha_{N})=2\sum_{0\leq i<k\leq N-1}\log\frac{1}{|a_{i}-a_{k}|}=2\sum_{k=1}^{N-1}\sum_{i=0}^{k-1}\log\frac{1}{|a_{k}-a_{i}|}=2\sum_{k=1}^{N-1}U_{k}(a_{k}).
\]
If $\tau(k)$ is the integer with the property
\begin{equation}\label{def:tau}
k=2^{n_{1}}+2^{n_{2}}+\cdots+2^{n_{\tau(k)}},\qquad n_{1}>n_{2}>\cdots> n_{\tau(k)}\geq 0,
\end{equation}
then according to \eqref{eq:minprod},
\[
U_{k}(a_{k})=-\log\, (2^{\tau(k)}),
\]
and therefore
\begin{equation}\label{energy}
E_{0}(\alpha_{N})=-2\,\log(2)\,\sum_{k=1}^{N-1}\tau(k).
\end{equation}

Note that $\tau(k)$ is the number of ones in the binary representation of $k$, so it satisfies the following properties:
\[
\tau(2^{n})=1,\qquad n\geq 0,
\]
and if $n_{1}>n_{2}>\cdots>n_{k}$, then
\begin{equation}\label{property}
\tau(2^{n_{1}}+2^{n_{2}}+\cdots+2^{n_{k}}+m)=k+\tau(m),\qquad 1\leq m\leq 2^{n_{k}}-1.
\end{equation}

Recall that the logarithmic energy of the configuration formed by $N$ equally spaced points in $S^{1}$ equals $-N\log N$. Since the configuration $\alpha_{2^{n}}$ consists of $2^{n}$ equally spaced points, we have
\begin{equation}\label{energyequalspace}
E_{0}(\alpha_{2^{n}})=-2^{n}\log(2^{n}).
\end{equation}
In particular, \eqref{energy} and \eqref{energyequalspace} give
\begin{equation}\label{sum}
n 2^{n-1}=\sum_{k=1}^{2^{n}-1}\tau(k).
\end{equation}

More generally, we have the following.

\begin{lemma}\label{lemmasum}
Assume that
\begin{equation}\label{eq:bin:N}
N=2^{n_{1}}+2^{n_{2}}+\cdots+2^{n_{t}},\qquad n_{1}>n_{2}>\cdots>n_{t}\geq 0.
\end{equation}
Then
\begin{equation}\label{sum:sk}
\sum_{k=1}^{N-1}\tau(k)=\sum_{i=1}^{t}\left(n_{i}+2(i-1)\right)\,2^{n_{i}-1}.
\end{equation}
\end{lemma}
\begin{proof}
The proof is by induction on $t$. If $t=1$ then \eqref{sum:sk} is exactly \eqref{sum}. Applying (\ref{property}) and (\ref{sum}) we obtain
\begin{align*}
\sum_{k=2^{n_{1}}+2^{n_{2}}+\cdots+2^{n_{t-1}}}^{2^{n_{1}}+2^{n_{2}}+\cdots+2^{n_{t}}-1} \tau(k) & =\tau(2^{n_{1}}+2^{n_{2}}+\cdots+2^{n_{t-1}})
+\sum_{m=1}^{2^{n_{t}}-1}\tau(2^{n_{1}}+2^{n_{2}}+\cdots+2^{n_{t-1}}+m)\\
 & =t-1+\sum_{m=1}^{2^{n_{t}}-1}(\tau(m)+t-1)=(t-1)\, 2^{n_{t}}+n_{t}\, 2^{n_{t}-1}.
\end{align*}
So \eqref{sum:sk} now follows easily by induction applying the previous computations and 
\[
\sum_{k=1}^{2^{n_{1}}+2^{n_{2}}+\cdots+2^{n_{t}}-1}\tau(k)
=\sum_{k=1}^{2^{n_{1}}+2^{n_{2}}+\cdots+2^{n_{t-1}}-1}\tau(k)
+\sum_{k=2^{n_{1}}+2^{n_{2}}+\cdots+2^{n_{t-1}}}^{2^{n_{1}}+2^{n_{2}}+\cdots+2^{n_{t}}-1}\tau(k).
\]
\end{proof}

\noindent \emph{Proof of Theorem~$\ref{theo:firstorderasymp}$.} From \eqref{sum:sk} and (\ref{energy}) it follows that if $N=2^{n_{1}}+2^{n_{2}}+\cdots+2^{n_{\tau(N)}}$, with $n_{1}>n_{2}>\cdots>n_{\tau(N)}\geq 0,$ then
\begin{equation}\label{energydecomp}
E_{0}(\alpha_{N})=\Lambda_{N,1}+\Lambda_{N,2},
\end{equation}
where
\begin{align*}
\Lambda_{N,1} & :=-\log(2)(n_{1} 2^{n_{1}}+n_{2} 2^{n_{2}}+\cdots+n_{\tau(N)} 2^{n_{\tau(N)}}),\\
\Lambda_{N,2} & :=-\log(2)(2^{n_{2}+1}+2\cdot 2^{n_{3}+1}+\cdots+(\tau(N)-1) 2^{n_{\tau(N)}+1}).
\end{align*}
We first justify that
\begin{equation}\label{limit1}
\lim_{N\rightarrow\infty}\frac{\Lambda_{N,2}}{N\log N}=0.
\end{equation}
Indeed, we have
\begin{gather*}
-\frac{1}{2\log(2)}\frac{\Lambda_{N,2}}{N\log N}=\frac{2^{n_{2}}+2\cdot 2^{n_{3}}+\cdots+(\tau(N)-1)2^{n_{\tau(N)}}}{(2^{n_{1}}+2^{n_{2}}+\cdots+2^{n_{\tau(N)}})\log N}\\
=\frac{2^{n_{2}-n_{1}}+2\cdot 2^{n_{3}-n_{1}}+\cdots+(\tau(N)-1)2^{n_{\tau(N)}-n_{1}}}{(1+2^{n_{2}-n_{1}}+\cdots+2^{n_{\tau(N)}-n_{1}})\log N}\\
\leq \frac{\frac{1}{2}+2\big(\frac{1}{2}\big)^{2}+\cdots+(\tau(N)-1)\big(\frac{1}{2}\big)^{\tau(N)-1}}{\log N}.
\end{gather*}
The numerator in the last expression is bounded by $\sum_{n=1}^{\infty}n 2^{-n}=2$ and (\ref{limit1}) follows.

We now show that
\begin{equation}\label{limit2}
\lim_{N\rightarrow\infty}\frac{\Lambda_{N,1}}{N\log N}=-1,
\end{equation}
hence \eqref{firstorderlimit} will follow from \eqref{energydecomp}, \eqref{limit1} and \eqref{limit2}.
We write
\begin{gather}
-\frac{\Lambda_{N,1}}{N\log N}=\frac{\log(2)(n_{1} 2^{n_{1}}+n_{2} 2^{n_{2}}+\cdots+n_{\tau(N)} 2^{n_{\tau(N)}})}{(2^{n_{1}}+2^{n_{2}}+\cdots+2^{n_{\tau(N)}})\log (2^{n_{1}}+2^{n_{2}}+\cdots+2^{n_{\tau(N)}})}\notag\\
=\frac{\log(2)(n_{1}+n_{2} 2^{n_{2}-n_{1}}+\cdots+n_{\tau(N)} 2^{n_{\tau(N)}-n_{1}})}{(1+2^{n_{2}-n_{1}}+\cdots+2^{n_{\tau(N)}-n_{1}})\log (2^{n_{1}}(1+2^{n_{2}-n_{1}}+\cdots+2^{n_{\tau(N)}-n_{1}}))}\label{expression:Lambdan1}\\
=\frac{\log(2)(n_{1}+n_{2} 2^{n_{2}-n_{1}}+\cdots+n_{\tau(N)} 2^{n_{\tau(N)}-n_{1}})}{(1+2^{n_{2}-n_{1}}+\cdots+2^{n_{\tau(N)}-n_{1}})\{n_{1}\log 2+\log (1+2^{n_{2}-n_{1}}+\cdots+2^{n_{\tau(N)}-n_{1}})\}}.\notag
\end{gather}
Since $n_{1}>n_{2}>\cdots>n_{\tau(N)}$, we have $1+2^{n_{2}-n_{1}}+\cdots+2^{n_{\tau(N)}-n_{1}}< \sum_{m=0}^{\infty}2^{-m}=2$. Therefore
\begin{equation}\label{normlimit}
\lim_{N\rightarrow\infty}\frac{(1+2^{n_{2}-n_{1}}+\cdots+2^{n_{\tau(N)}-n_{1}})\log(1+2^{n_{2}-n_{1}}+\cdots+2^{n_{\tau(N)}-n_{1}})}{ n_{1}+n_{2} 2^{n_{2}-n_{1}}+\cdots+n_{\tau(N)} 2^{n_{\tau(N)}-n_{1}}}=0,
\end{equation}
due to the fact that $n_{1}\rightarrow\infty$ as $N\rightarrow\infty$. 

Now we write
\begin{gather}
\frac{n_{1}\log (2)\, (1+2^{n_{2}-n_{1}}+\cdots+2^{n_{\tau(N)}-n_{1}})}{\log (2)\, (n_{1}+n_{2} 2^{n_{2}-n_{1}}+\cdots+n_{\tau(N)} 2^{n_{\tau(N)}-n_{1}})}\notag\\
=\frac{1+2^{n_{2}-n_{1}}+\cdots+2^{n_{\tau(N)}-n_{1}}}{1+\big(\frac{n_{2}}{n_{1}}\big)\,2^{n_{2}-n_{1}}+\cdots+\big(\frac{n_{\tau(N)}}{n_{1}}\big)\,2^{n_{\tau(N)}-n_{1}}}=:\frac{c_{N}}{d_{N}}.\label{def:anbn}
\end{gather}
In order to prove that $c_{N}/d_{N}\rightarrow 1$ it suffices to show that $c_{N}-d_{N}\rightarrow 0$. We have
\[
c_{N}-d_{N}=\left(1-\frac{n_{2}}{n_{1}}\right)\,2^{n_{2}-n_{1}}+\left(1-\frac{n_{3}}{n_{1}}\right)\,2^{n_{3}-n_{1}}
+\cdots+\left(1-\frac{n_{\tau(N)}}{n_{1}}\right)\,2^{n_{\tau(N)}-n_{1}}.
\]
One can prove that this expression approaches zero applying Lebesgue's dominated convergence theorem. On $\mathbb{N}=\{1,2,\ldots\}$ we define a sequence of functions $(f_{N})_{N}$ as follows:
\[
f_{N}(m)=\begin{cases}
(1-\frac{n_{i}}{n_{1}})\, 2^{n_{i}-n_{1}} & \mbox{if}\,\,m=n_{1}-n_{i}\,\,\mbox{for some $i$, $2\leq i\leq \tau(N)$},\\
0 & \mbox{otherwise}.
\end{cases}
\]
The function $f_{N}$ is well-defined, and clearly
\[
c_{N}-d_{N}=\sum_{m=1}^{\infty} f_{N}(m).
\]
For each fixed $m$, $f_{N}(m)=0$ or $f_{N}(m)=(m/n_{1})\, 2^{-m}$. In any case, since $n_{1}\rightarrow\infty$ we have $f_{N}(m)\rightarrow 0$ as $N\rightarrow\infty$. Hence Lebesgue's theorem gives
\[
\lim_{N\rightarrow\infty} c_{N}-d_{N}=\sum_{m=1}^{\infty} 0=0.
\]
Since $c_{N}/d_{N}\rightarrow 1$, \eqref{limit2} follows from \eqref{expression:Lambdan1}, \eqref{normlimit} and \eqref{def:anbn}. \hfill $\Box$

\begin{remark}
We would like to emphasize that if $(x_{n})_{n=0}^{\infty}$ is any sequence of pairwise distinct points in $S^{1}$, then it is clear that for the sequence of configurations $\omega_{N}=\{x_{0},x_{1},\ldots,x_{N-1}\}$ we have
\begin{equation}\label{ulgeneralseq}
\limsup_{N\rightarrow\infty}\frac{E_{0}(\omega_{N})}{N\log N}\geq -1.
\end{equation}
If we have equality in \eqref{ulgeneralseq}, then
\begin{equation}\label{lgeneralseq}
\lim_{N\rightarrow\infty}\frac{E_{0}(\omega_{N})}{N\log N}=-1
\end{equation}
and the sequence $(x_{n})_{n=0}^{\infty}$ will be asymptotically uniformly distributed; that is, we have the weak-star convergence
\begin{equation}\label{weakstarconvunif}
\frac{1}{N}\sum_{x\in\omega_{N}}\delta_{x}\xrightarrow[N\rightarrow\infty]{*}\sigma, 
\end{equation}
where $\sigma$ denotes the normalized arc length measure on $S^{1}$. Indeed, if \eqref{lgeneralseq} holds, then
\[
\lim_{N\rightarrow\infty}\frac{E_{0}(\omega_{N})}{N^{2}}=0=\iint\log\frac{1}{|x-y|}\,d\sigma(x)\,d\sigma(y),
\]
and this implies \eqref{weakstarconvunif} by a standard argument in potential theory, see \cite{Ran}.
\end{remark}

\section{Second-order estimates in the logarithmic case}\label{section:secondorderest}

\noindent\emph{Proof of Theorem~$\ref{theo:secondorderest}$.} The inequality on the left-hand side of \eqref{secondorderest} is obvious. If $N$ has the binary representation \eqref{eq:bin:N}, then in virtue of \eqref{energy} and \eqref{sum:sk} we have
\begin{equation}\label{exp:modenergy}
\frac{E_{0}(\alpha_{N})+N\log N}{N}=-\log (2)\,\frac{\sum_{i=1}^{t}(n_{i}+2i-2)\,2^{n_{i}}}{\sum_{i=1}^{t}2^{n_{i}}}+\log\left(\sum_{i=1}^{t}2^{n_{i}}\right),
\end{equation}
hence the inequality on the right-hand side of \eqref{secondorderest} is the same as
\[
-\log (2)\,\frac{\sum_{i=1}^{t}(n_{i}+2i-2)\,2^{n_{i}}}{\sum_{i=1}^{t}2^{n_{i}}}+\log\left(\sum_{i=1}^{t}2^{n_{i}}\right)<2\log 2-\log 3.
\]
Simplifying we obtain that this is equivalent to
\[
\log\left(3\sum_{i=1}^{t}2^{n_{i}}\right)<\log(2)\,\frac{\sum_{i=1}^{t}(n_{i}+2i) 2^{n_{i}}}{\sum_{i=1}^{t}2^{n_{i}}}.
\]
Hence we want to show that
\begin{equation}\label{ineqsecondorder}
3\sum_{i=1}^{t} 2^{n_{i}}<2^{c_{N}},\qquad c_{N}=\frac{\sum_{i=1}^{t}(n_{i}+2i)\, 2^{n_{i}}}{\sum_{i=1}^{t}2^{n_{i}}}.
\end{equation}
In order to prove \eqref{ineqsecondorder} we apply the following inequality, which can be found in \cite[page 78]{HLP}: For any collection of positive numbers $(b_{i})_{i=1}^{t}$ and $(p_{i})_{i=1}^{t}$ we have
\begin{equation}\label{ineqHLP}
\frac{\sum_{i=1}^{t}p_{i}\,b_{i}}{\sum_{i=1}^{t}p_{i}}\leq \exp\left(\frac{\sum_{i=1}^{t}p_{i}\,b_{i}\,\log(b_{i})}{\sum_{i=1}^{t}p_{i}\,b_{i}}\right),
\end{equation}
with equality only if all the $b$'s are equal. The inequality \eqref{ineqHLP} is obtained applying Jensen's inequality to the convex function $x\log x$. Taking in \eqref{ineqHLP} the values
\[
p_{i}=2^{-2i},\qquad b_{i}=2^{n_{i}+2i},\qquad i=1,\ldots,t,
\]
we obtain after simplification the expression
\[
\frac{\sum_{i=1}^{t}2^{n_{i}}}{\sum_{i=1}^{t} 4^{-i}}<2^{c_{N}},
\]
which gives \eqref{ineqsecondorder}.

In order to prove \eqref{ulsecondorder}, it suffices now to show that for the subsequence
\begin{equation}\label{choiceN}
N=N(k)=\sum_{j=0}^{k} 4^{j}=\frac{4^{k+1}-1}{3}
\end{equation}
one gets
\[
\lim_{k\rightarrow\infty}\frac{E_{0}(\alpha_{N})+N\log N}{N}=\log\left(\frac{4}{3}\right).
\]
For this it is convenient to rewrite \eqref{exp:modenergy} as
\begin{equation}\label{modenergyrewrite}
\frac{E_{0}(\alpha_{N})+N\log N}{N}=\log(2)\,\frac{\sum_{i=2}^{t}(n_{1}-n_{i}+2-2i)\,2^{n_{i}-n_{1}}}{\sum_{i=1}^{t}2^{n_{i}-n_{1}}}+\log\left(\sum_{i=1}^{t}2^{n_{i}-n_{1}}\right).
\end{equation}
For the choice \eqref{choiceN} of $N$ we have $n_{1}-n_{i}=2i-2$, hence the first term vanishes, while the second term approaches $\log(4/3)$.\hfill $\Box$

\bigskip

An interesting property of the sequence analyzed in Theorem \ref{theo:secondorderest} is the fact that 
\[
\frac{E_{0}(\alpha_{N})+N\log N}{N}=\frac{E_{0}(\alpha_{2N})+2N\log (2N)}{2N},\qquad \mbox{for all}\,\,N\geq 1,
\]
which can be easily checked using \eqref{modenergyrewrite}. This property explains the ``periodic" behavior of the sequence $\left(\frac{E_{0}(\alpha_{N})+N\log N}{N}\right)_{N}$ that can be observed in Fig.~\ref{fig:energylog} below, with increasing ``periods" of length $2^{n}$.

\begin{figure}[h]
\centering
\includegraphics[totalheight=3in,keepaspectratio]{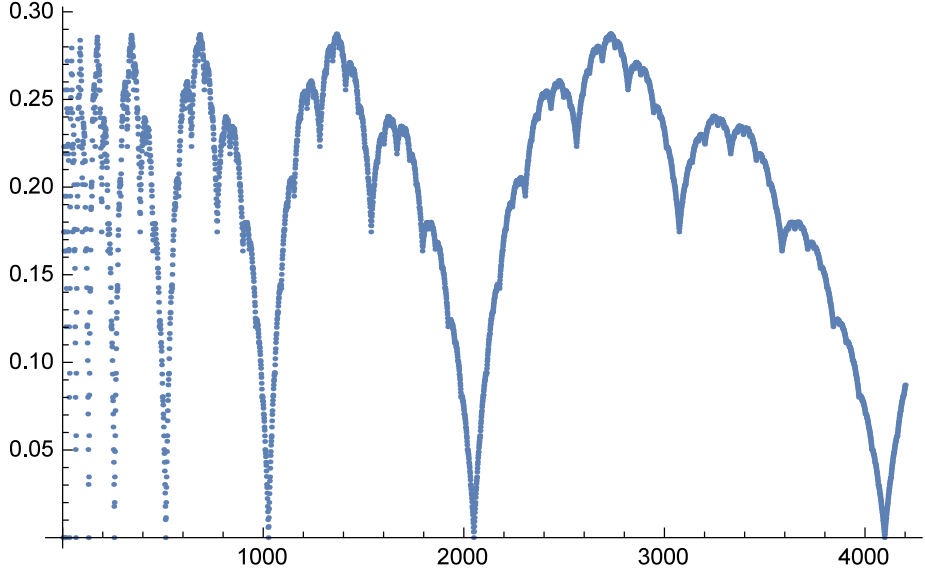}
\caption{This figure shows the first $4200$ points of the sequence $\left(E_{0}(\alpha_{N})+N\log N)/N\right)_{N}$.}
\label{fig:energylog}
\end{figure}

\section{Proofs of results in the Riesz setting}\label{section:RieszenergyLeja}

We begin with a formula that expresses the Riesz $s$-energy of the first $N$ points $\alpha_{N}$ of a greedy $s$-energy sequence on the unit circle in terms of the binary representation of $N$.  

\begin{proposition}\label{prop:Reiszenergybinary}
Let $(a_{n})_{n=0}^{\infty}$ be a greedy $s$-energy sequence on $S^{1}$, and let $\alpha_{N}=(a_{n})_{n=0}^{N-1}$. Assume that $N$ has the binary representation \eqref{eq:bin:N}. Then
\begin{equation}\label{Rieszenergybinary}
E_{s}(\alpha_{N})=\sum_{k=1}^{t-1}\left(\sum_{j=k+1}^{t} 2^{n_{j}-n_{k}}\right)\mathcal{L}_{s}(2^{n_{k}+1})+\sum_{k=1}^{t}\left(1-\sum_{j=k+1}^{t} 2^{n_{j}-n_{k}+1}\right)\mathcal{L}_{s}(2^{n_{k}}),
\end{equation}
understanding $\sum_{t+1}^{t}$ as empty sum.
\end{proposition}
\begin{proof}
The proof of \eqref{Rieszenergybinary} is obtained from a repeated application of the following simple property. If $A$ and $B$ are two finite sets of points on the unit circle with $A\cap B=\emptyset$, then
\begin{equation}\label{Rieszdecompformula}
E_{s}(A\cup B)=E_{s}(A)+E_{s}(B)+2\sum_{y\in B}\sum_{x\in A}|x-y|^{-s}.
\end{equation}
Let $(a_{n})_{n=0}^\infty$ be a greedy $s$-energy sequence on $S^{1}$. Recall that this sequence also has the structure of a Leja sequence, and hence the properties $1)$--$3)$ described in Sec.~\ref{section:intro} are also applicable for this sequence. Let $A_{1}$ denote the set formed by the first $2^{n_{1}}$ points of the sequence and $B_{1}$ denote the next $N-2^{n_{1}}=2^{n_{2}}+\cdots+2^{n_{t}}$ points of the sequence, i.e., 
\[
A_{1}:=(a_{n})_{n=0}^{2^{n_{1}-1}},\qquad B_{1}:=(a_{n})_{n=2^{n_{1}}}^{N-1}.
\]
Since the points in $A_{1}$ are equally spaced, we have $E_{s}(A_{1})=\mathcal{L}_{s}(2^{n_{1}})$. Any $y\in B_{1}$ lies in the midpoint of one of the $2^{n_{1}}$ arcs determined by the points of $A_{1}$. So clearly $\sum_{x\in A_{1}}|x-y|^{-s}$ is independent of $y$, and we can write this expression as the difference
\begin{align*}
\sum_{x\in A_{1}}|x-y|^{-s}= & 2^{-s}\sum_{j=1}^{2^{n_{1}+1}-1}\left(\sin \frac{\pi j}{2^{n_{1}+1}}\right)^{-s}-2^{-s}\sum_{j=1}^{2^{n_{1}}-1}\left(\sin \frac{\pi j}{2^{n_{1}}}\right)^{-s}\\
= & \frac{1}{2^{n_{1}+1}}\,\mathcal{L}_{s}(2^{n_{1}+1})-\frac{1}{2^{n_{1}}}\,\mathcal{L}_{s}(2^{n_{1}}).
\end{align*}
We conclude from \eqref{Rieszdecompformula} and the computation above that 
\[
E_{s}(\alpha_{N})=\left(\sum_{j=2}^{t} 2^{n_{j}-n_{1}}\right)\mathcal{L}_{s}(2^{n_{1}+1})+\left(1-\sum_{j=2}^{t}2^{n_{j}-n_{1}+1}\right)\mathcal{L}_{s}(2^{n_{1}})+E_{s}(B_{1}).
\]

Now we can apply this argument to the set $B_{1}$, since this set itself has the structure of the first $N-2^{n_{1}}$ points of a greedy sequence, see \cite[Theorem 5]{BiaCal} and \cite[Lemma 4.2]{LopSaff}. In particular, if we make the partition $B_{1}=A_{2}\cup B_{2}$, where $A_{2}$ is the set formed by the first $2^{n_{2}}$ points in $B_{1}$ and $B_{2}$ is the set formed by the remaining $N-2^{n_{1}}-2^{n_{2}}=2^{n_{3}}+\cdots+2^{n_{t}}$ points in $B_{1}$, then again we have that $A_{2}$ is formed by equally spaced points and any point of $B_{2}$ lies in the midpoint of one of the $2^{n_{2}}$ arcs determined by the points of $A_{2}$. Hence as before we get
\[
E_{s}(B_{1})=\left(\sum_{j=3}^{t} 2^{n_{j}-n_{2}}\right)\mathcal{L}_{s}(2^{n_{2}+1})+\left(1-\sum_{j=3}^{t}2^{n_{j}-n_{2}+1}\right)\mathcal{L}_{s}(2^{n_{2}})+E_{s}(B_{2}),
\]
and so
\[
E_{s}(\alpha_{N})=\sum_{k=1}^{2}\left(\sum_{j=k+1}^{t} 2^{n_{j}-n_{k}}\right)\mathcal{L}_{s}(2^{n_{k}+1})+\sum_{k=1}^{2}\left(1-\sum_{j=k+1}^{t} 2^{n_{j}-n_{k}+1}\right)\mathcal{L}_{s}(2^{n_{k}})+E_{s}(B_{2}).
\]
Applying this argument repeatedly it is clear that we arrive at \eqref{Rieszenergybinary}. 
\end{proof}

Before giving the proofs of the results in the Riesz setting, we make some remarks concerning the sets $\Theta_{p}$ and the functions $H$ defined in \eqref{def:H}. 

The reader can easily check that an alternative way to define the set $\Theta_{p}$ is the following. This set consists of all vectors $\vec{\theta}=(\theta_{1},\ldots,\theta_{p})$ that can be written in the form
\begin{equation}\label{def:vectorsinThetap}
\vec{\theta}=\left(\frac{2^{t_{1}}}{M}, \frac{2^{t_{2}}}{M},\ldots, \frac{2^{t_{r-1}}}{M}, \frac{1}{M}, 0,\ldots, 0\right),
\end{equation}
where $M=2^{t_{1}}+2^{t_{2}}+\cdots+2^{t_{r-1}}+1$ is an odd integer with $t_{1}>t_{2}>\cdots>t_{r-1}>0$ and $1\leq r\leq p$. The number of zeros in \eqref{def:vectorsinThetap} is then $p-r$, if they appear. In particular we see that the set $\Theta_{p}$ can be regarded as a subset of $\Theta_{p+1}$, for all $p$. We preferred to give the Definition~\ref{def:Thetap} for $\Theta_{p}$ instead of the one described here since we are only going to make use of the limiting property \eqref{eq:conddefthetavec}.

It follows from \eqref{def:H} that if $\vec{\theta}=(\theta_{1},\ldots,\theta_{p})$ satisfies the condition $\theta_{k}\geq 2 \sum_{j=k+1}^{p}\theta_{j}$ for all $k=1,\ldots,p-1$, then $H(\vec{\theta};s)$ is convex as a function of $s$ since in this case we can write it as a positive linear combination of convex functions.

\subsection{Second-order asymptotics in the Riesz case for $0<s<1$}

Below we will make use of a fortunate relation between the coefficients appearing in \eqref{Rieszenergybinary}, the arguments of $\mathcal{L}_{s}$ in this formula, and $N^{2}$. The reader can easily check that for $N$ as in \eqref{eq:bin:N} we have
\begin{equation}\label{expansionN2}
N^{2}=\sum_{k=1}^{t-1}\left(\sum_{j=k+1}^{t} 2^{n_{j}-n_{k}}\right)\,2^{2(n_{k}+1)}+\sum_{k=1}^{t}\left(1-\sum_{j=k+1}^{t} 2^{n_{j}-n_{k}+1}\right)\,2^{2 n_{k}}.
\end{equation}
So if we introduce the notation
\begin{equation}\label{def:Rs}
\mathcal{R}_{s}(N):=\frac{\mathcal{L}_{s}(N)-I_{s}(\sigma) N^{2}}{N^{1+s}},\qquad 0<s<1,
\end{equation}
it follows from \eqref{Rieszenergybinary}and \eqref{expansionN2} that
\begin{align}
\frac{E_{s}(\alpha_{N})-I_{s}(\sigma) N^{2}}{N^{1+s}}= & \sum_{k=1}^{t-1}\left(\frac{2^{n_{k}+1}}{N}\right)^{1+s}\left(\sum_{j=k+1}^{t}2^{n_{j}-n_{k}}\right)\mathcal{R}_{s}(2^{n_{k}+1})\notag\\
+ & \sum_{k=1}^{t}\left(\frac{2^{n_{k}}}{N}\right)^{1+s}\left(1-\sum_{j=k+1}^{t}2^{n_{j}-n_{k}+1}\right)\mathcal{R}_{s}(2^{n_{k}}).\label{eq:secorderRieszenergy}
\end{align}

See an illustration of the sequence $\left(\frac{E_{s}(\alpha_{N})-I_{s}(\sigma) N^{2}}{N^{1+s}}\right)_{N}$ in Fig.~\ref{fig:energyonehalf} below.

\begin{figure}[h]
\centering
\includegraphics[totalheight=3in,keepaspectratio]{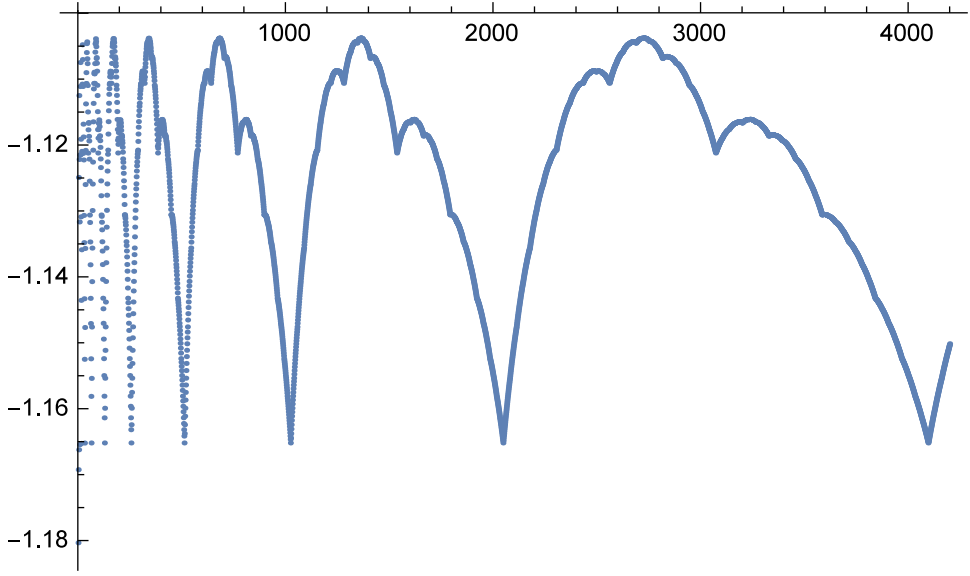}
\caption{This is a plot of the first $4200$ points of the sequence $\left((E_{s}(\alpha_{N})-I_{s}(\sigma) N^{2})/N^{1+s}\right)_{N}$ in the case $s=1/2$.}
\label{fig:energyonehalf}
\end{figure}

\bigskip

\noindent\emph{Proof of Theorem~$\ref{theo:secorderasympRieszenergy:1}$.}
We first prove the inequality ``$\geq$" in \eqref{eq:secorderasympRieszenergy:1}, which is straightforward. Let $p\in\mathbb{N}$ be arbitrary and fix a vector $\vec{\theta}=(\theta_{1},\ldots,\theta_{p})\in\Theta_{p}$. By Definition~\ref{def:Thetap}, there exists an infinite sequence $\mathcal{N}$ of integers of the form $N=2^{n_{1}}+2^{n_{2}}+\cdots+2^{n_{p}}$, $n_{1}>n_{2}>\cdots>n_{p}\geq 0$ such that \eqref{eq:conddefthetavec} holds. Applying \eqref{eq:secorderRieszenergy} we have
\begin{gather*}
\frac{E_{s}(\alpha_{N})-I_{s}(\sigma) N^{2}}{N^{1+s}}\\
=\sum_{k=1}^{p-1}\left(\frac{2^{n_{k}+1}}{N}\right)^{1+s}\left(\sum_{j=k+1}^{p}2^{n_{j}-n_{k}}\right)\mathcal{R}_{s}(2^{n_{k}+1})+\sum_{k=1}^{p}\left(\frac{2^{n_{k}}}{N}\right)^{1+s}\left(1-\sum_{j=k+1}^{p}2^{n_{j}-n_{k}+1}\right)\mathcal{R}_{s}(2^{n_{k}})\\
=\sum_{k=1}^{p-1}\left(\frac{2^{n_{k}+1}}{N}\right)^{s}\left(\sum_{j=k+1}^{p}\frac{2^{n_{j}+1}}{N}\right)\mathcal{R}_{s}(2^{n_{k}+1})+\sum_{k=1}^{p}\left(\frac{2^{n_{k}}}{N}\right)^{s}\left(\frac{2^{n_{k}}}{N}-\sum_{j=k+1}^{p}\frac{2^{n_{j}+1}}{N}\right)\mathcal{R}_{s}(2^{n_{k}}).
\end{gather*}
Using now \eqref{eq:conddefthetavec}, \eqref{def:Rs} and \eqref{eq:secorderasympRieszFekete}, we get
\begin{equation}\label{eq:limitsubseq}
\lim_{N\in\mathcal{N}}\frac{E_{s}(\alpha_{N})-I_{s}(\sigma) N^{2}}{N^{1+s}}
=\left(\sum_{k=1}^{p-1}(2\theta_{k})^{s} \sum_{j=k+1}^{p}2\theta_{j}+\sum_{k=1}^{p}\theta_{k}^{s}\,\Large(\theta_{k}-\sum_{j=k+1}^{p}2\theta_{j}\Large)\right)\frac{2\zeta(s)}{(2\pi)^{s}}.
\end{equation}
Here we have taken into account that if for some particular $k=1,\ldots,p$, the sequence $2^{n_{k}}$ does not approach infinity, then $2^{n_{k}}/N$ approaches $\theta_{k}=0$ and therefore we still have $(2^{n_{k}}/N)^{s}\,\mathcal{R}_{s}(2^{n_{k}})\rightarrow \theta_{k}^{s}\, 2\zeta(s)/(2\pi)^{s}=0$. The first factor on the right-hand side of \eqref{eq:limitsubseq} is exactly $H(\vec{\theta};s)$, and therefore
\[
\limsup_{N\rightarrow\infty}\frac{E_{s}(\alpha_{N})-I_{s}(\sigma) N^{2}}{N^{1+s}}\geq H(\vec{\theta};s)\,\frac{2\zeta(s)}{(2\pi)^{s}}.
\]
Since $p$ and $\vec{\theta}$ were arbitrary, we obtain the desired inequality.

Now we prove the converse inequality in \eqref{eq:secorderasympRieszenergy:1}. Let $\mathcal{N}\subset\mathbb{N}$ be an infinite sequence for which the sequence $\left(\frac{E_{s}(\alpha_{N})-I_{s}(\sigma) N^{2}}{N^{1+s}}\right)_{N\in\mathcal{N}}$ converges, and we shall show that
\begin{equation}\label{eq:ineqtobeproved}
\lim_{N\in\mathcal{N}}\frac{E_{s}(\alpha_{N})-I_{s}(\sigma) N^{2}}{N^{1+s}}\leq \underline{h}(s)\,\frac{2\zeta(s)}{(2\pi)^{s}}.
\end{equation}

Assume first that there exists $p\geq 1$ such that an infinite number of integers $N\in\mathcal{N}$ satisfy the property $\tau(N)=p$, cf. \eqref{def:tau}. Then, taking a subsequence $\widetilde{\mathcal{N}}$ of $\mathcal{N}$ if necessary, such that the integers $N=2^{n_{1}}+2^{n_{2}}+\cdots+2^{n_{p}}\in\widetilde{\mathcal{N}}$ satisfy 
\[
\lim_{N\in\widetilde{N}}\frac{2^{n_{i}}}{N}=\theta_{i},\qquad \mbox{for all}\,\,i=1,\ldots,p,
\]
we get
\[
\lim_{N\in\mathcal{N}}\frac{E_{s}(\alpha_{N})-I_{s}(\sigma) N^{2}}{N^{1+s}}=\lim_{N\in\widetilde{\mathcal{N}}}\frac{E_{s}(\alpha_{N})-I_{s}(\sigma) N^{2}}{N^{1+s}}=H((\theta_{1},\ldots,\theta_{p});s)\,\frac{2\zeta(s)}{(2\pi)^{s}},
\]
and therefore \eqref{eq:ineqtobeproved} holds.

So let us assume now that such an integer $p$ does not exist. This means that we assume now that $\tau(N)\rightarrow\infty$ as $N\rightarrow\infty$ in the sequence $\mathcal{N}$. Let us rewrite, for $N=2^{n_{1}}+2^{n_{2}}+\cdots+2^{n_{\tau(N)}}\in\mathcal{N}$, $n_{1}>n_{2}>\cdots>n_{\tau(N)}\geq 0$, the expression 
\begin{gather}
\frac{E_{s}(\alpha_{N})-I_{s}(\sigma) N^{2}}{N^{1+s}}\notag\\
=\sum_{k=1}^{\tau(N)}\left(\frac{2^{n_{k}}}{N}\right)^{1+s}\left[2^{s+1}\,\mathcal{R}_{s}(2^{n_{k}+1})\,\sum_{j=k+1}^{\tau(N)}2^{n_{j}-n_{k}}+\mathcal{R}_{s}(2^{n_{k}})\left(1-\sum_{j=k+1}^{\tau(N)} 2^{n_{j}-n_{k}+1}\right)\right].\label{eq:rewritingsecorderenergy}
\end{gather}
and let us introduce the notation
\begin{equation}\label{def:lambdaNk}
\lambda_{N,k}:=2^{s+1}\,\mathcal{R}_{s}(2^{n_{k}+1})\,\sum_{j=k+1}^{\tau(N)}2^{n_{j}-n_{k}}+\mathcal{R}_{s}(2^{n_{k}})\left(1-\sum_{j=k+1}^{\tau(N)} 2^{n_{j}-n_{k}+1}\right).
\end{equation}
Since the sequence $(\mathcal{R}_{s}(N))_{N}$ is bounded, it is evident that there exists an absolute constant $C_{1}>0$ independent of $N$, such that
\begin{equation}\label{bound1}
\left|\lambda_{N,k}\right|\leq C_{1},\qquad \mbox{for all}\,\,N\in\mathcal{N}\,\,\mbox{and}\,\,k=1,\ldots,\tau(N).
\end{equation}
On the other hand, we have the following simple estimate for each $N=2^{n_{1}}+\cdots+2^{n_{\tau(N)}}\in\mathcal{N}$,
\begin{equation}\label{bound2}
\frac{2^{n_{k}}}{N}\leq \frac{2^{n_{k}}}{2^{n_{1}}}=2^{n_{k}-n_{1}}\leq 2^{-(k-1)},\qquad k=1,\ldots,\tau(N).
\end{equation}
Now let $0<\epsilon<1$ be fixed. It follows from \eqref{bound2} that there exists $M=M(\epsilon)\in\mathbb{N}$ independent of $N$ such that
\begin{equation}\label{bound4}
\sum_{k=M+1}^{\tau(N)}\frac{2^{n_{k}}}{N}<\epsilon,\qquad \mbox{for all}\,\,N\in\mathcal{N}.
\end{equation}
hence \eqref{bound1} and \eqref{bound4} give 
\begin{equation}\label{bound3}
\sum_{k=M+1}^{\tau(N)}\left(\frac{2^{n_{k}}}{N}\right)^{1+s}|\lambda_{N,k}|<C_{1}\epsilon,\qquad \mbox{for all}\,\,N\in\mathcal{N}.
\end{equation}
Applying \eqref{eq:rewritingsecorderenergy} and \eqref{def:lambdaNk} we can write
\begin{equation}\label{energydecomp:1}
\frac{E_{s}(\alpha_{N})-I_{s}(\sigma) N^{2}}{N^{1+s}}=S_{N,M,1}+S_{N,M,2},
\end{equation}
where
\begin{equation}\label{def:SNM1SNM2}
S_{N,M,1}:=\sum_{k=1}^{M}\left(\frac{2^{n_{k}}}{N}\right)^{1+s}\lambda_{N,k},\qquad
S_{N,M,2}:=\sum_{k=M+1}^{\tau(N)}\left(\frac{2^{n_{k}}}{N}\right)^{1+s}\lambda_{N,k},
\end{equation}
hence by \eqref{bound3} we have
\begin{equation}\label{bound5}
|S_{N,M,2}|\leq C_{1}\epsilon,\qquad \mbox{for all}\,\,N\in\mathcal{N}.
\end{equation}

Now we focus on the sum $S_{N,M,1}$. First we rewrite $\lambda_{N,k}$ in the form
\begin{gather}
\lambda_{N,k}=\mathcal{R}_{s}(2^{n_{k}})+2^{-n_{k}+1}\,(2^{s}\,\mathcal{R}_{s}(2^{n_{k}+1})-\mathcal{R}_{s}(2^{n_{k}}))\,\sum_{j=k+1}^{\tau(N)}2^{n_{j}}\label{def:lambdaNk:2}\\
=\mathcal{R}_{s}(2^{n_{k}})+2^{-n_{k}+1}\,(2^{s}\,\mathcal{R}_{s}(2^{n_{k}+1})-\mathcal{R}_{s}(2^{n_{k}}))\,\left(\sum_{j=k+1}^{M}2^{n_{j}}+\sum_{j=M+1}^{\tau(N)}2^{n_{j}}\right).\notag
\end{gather}
This shows that we can write
\begin{equation}\label{energydecomp:2}
S_{N,M,1}=D_{N,M,1}+D_{N,M,2},
\end{equation}
where
\begin{align}
D_{N,M,1} & :=\sum_{k=1}^{M}\left(\frac{2^{n_{k}}}{N}\right)^{1+s}\left[\mathcal{R}_{s}(2^{n_{k}})+2^{-n_{k}+1}\,(2^{s}\,\mathcal{R}_{s}(2^{n_{k}+1})-\mathcal{R}_{s}(2^{n_{k}}))\,\sum_{j=k+1}^{M}2^{n_{j}}\right],\label{def:DNM1}\\
D_{N,M,2} & :=\left(\sum_{j=M+1}^{\tau(N)} 2^{n_{j}}\right) \sum_{k=1}^{M}\left(\frac{2^{n_{k}}}{N}\right)^{1+s}
2^{-n_{k}+1}\,(2^{s}\,\mathcal{R}_{s}(2^{n_{k}+1})-\mathcal{R}_{s}(2^{n_{k}})). \label{def:DNM2}
\end{align}
Let's first estimate the sum \eqref{def:DNM2}. We have
\begin{gather*}
D_{N,M,2}=\left(\sum_{j=M+1}^{\tau(N)}\frac{2^{n_{j}}}{N}\right)\sum_{k=1}^{M}\left(\frac{2^{n_{k}}}{N}\right)^{s} (2^{s+1}\,\mathcal{R}_{s}(2^{n_{k}+1})-2\,\mathcal{R}_{s}(2^{n_{k}})).
\end{gather*}
Using \eqref{bound2} and the boundedness of the sequence $(\mathcal{R}_{s}(N))_{N}$, we find that there exists an absolute constant $C_{2}>0$ such that
\[
\left|\sum_{k=1}^{M}\left(\frac{2^{n_{k}}}{N}\right)^{s} (2^{s+1}\,\mathcal{R}_{s}(2^{n_{k}+1})-2\,\mathcal{R}_{s}(2^{n_{k}}))\right|<C_{2},\qquad \mbox{for all}\,\,N\in\mathcal{N}.
\]
This estimate and \eqref{bound4} shows that
\begin{equation}\label{bound6}
|D_{N,M,2}|<C_{2}\epsilon,\qquad \mbox{for all}\,\,N\in\mathcal{N}.
\end{equation}
Finally we analyze the sum \eqref{def:DNM1}. Introducing the notation
\[
\widetilde{\lambda}_{N,k}:=\mathcal{R}_{s}(2^{n_{k}})+2^{-n_{k}+1}\,(2^{s}\,\mathcal{R}_{s}(2^{n_{k}+1})-\mathcal{R}_{s}(2^{n_{k}}))\,\sum_{j=k+1}^{M}2^{n_{j}},
\]
we can write
\begin{equation}\label{energydecomp:3}
D_{N,M,1}=\sum_{k=1}^{M}\left(\frac{2^{n_{k}}}{N}\right)^{1+s} \widetilde{\lambda}_{N,k}=E_{N,M,1}+E_{N,M,2},
\end{equation}
where
\begin{align*}
E_{N,M,1} & :=\sum_{k=1}^{M}\left(\frac{2^{n_{k}}}{2^{n_{1}}+2^{n_{2}}+\cdots+2^{n_{M}}}\right)^{1+s}\widetilde{\lambda}_{N,k},\\
E_{N,M,2} & :=\left(\left(\frac{2^{n_{1}}+2^{n_{2}}+\cdots+2^{n_{M}}}{N}\right)^{1+s}-1\right) \sum_{k=1}^{M}\left(\frac{2^{n_{k}}}{2^{n_{1}}+2^{n_{2}}+\cdots+2^{n_{M}}}\right)^{1+s} \widetilde{\lambda}_{N,k}.
\end{align*}
Again the numbers $\widetilde{\lambda}_{N,k}$ are uniformly bounded and we have
\begin{gather*}
\sum_{k=1}^{M}\left(\frac{2^{n_{k}}}{2^{n_{1}}+2^{n_{2}}+\cdots+2^{n_{M}}}\right)^{1+s}\leq \sum_{k=1}^{M}\frac{2^{n_{k}}}{2^{n_{1}}+2^{n_{2}}+\cdots+2^{n_{M}}}=1,\\
\left|\left(\frac{2^{n_{1}}+2^{n_{2}}+\cdots+2^{n_{M}}}{N}\right)^{1+s}-1\right|\leq 1-(1-\epsilon)^{1+s},
\end{gather*}
where in the latter inequality we used \eqref{bound4}. We conclude that
\begin{equation}\label{bound7}
|E_{N,M,2}|\leq C_{3}\, (1-(1-\epsilon)^{1+s}),
\end{equation}
for some constant $C_{3}>0$.

Note that the expression $E_{N,M,1}$ is exactly as in \eqref{eq:rewritingsecorderenergy} but with $N$ replaced by $2^{n_{1}}+\cdots+2^{n_{M}}$ and $\tau(N)$ replaced by $M$. Therefore, as before we can find a subsequence $\widetilde{\mathcal{N}}$ of $\mathcal{N}$ such that
\begin{equation}\label{cond:subseqNtilde}
\lim_{N\in\widetilde{\mathcal{N}}}\frac{2^{n_{i}}}{2^{n_{1}}+\cdots+2^{n_{M}}}=\theta_{i},\qquad \mbox{for all}\,\,i=1,\ldots,M,
\end{equation}
and consequently
\begin{equation}\label{finallimit}
\lim_{N\in\widetilde{\mathcal{N}}} E_{N,M,1}=H((\theta_{1},\ldots,\theta_{M});s)\,\frac{2\zeta(s)}{(2\pi)^{s}}.
\end{equation}
Applying now the relations \eqref{energydecomp:1}, \eqref{energydecomp:2}, \eqref{energydecomp:3} and the bounds \eqref{bound5}, \eqref{bound6}, \eqref{bound7} and \eqref{finallimit}, we conclude that
\begin{gather*}
\lim_{N\in\mathcal{N}}\frac{E_{s}(\alpha_{N})-I_{s}(\sigma) N^{2}}{N^{1+s}}
=\lim_{N\in\widetilde{\mathcal{N}}}\frac{E_{s}(\alpha_{N})-I_{s}(\sigma) N^{2}}{N^{1+s}}=\lim_{N\in\widetilde{\mathcal{N}}}\left(E_{N,M,1}+E_{N,M,2}+D_{N,M,2}+S_{N,M,2}\right)\\
\leq H((\theta_{1},\ldots,\theta_{M});s)\,\frac{2\zeta(s)}{(2\pi)^{s}}
+C_{3}(1-(1-\epsilon)^{1+s})+C_{2}\epsilon+C_{1}\epsilon\\
\leq \underline{h}(s)\,\frac{2\zeta(s)}{(2\pi)^{s}}+C_{3}(1-(1-\epsilon)^{1+s})+C_{2}\epsilon+C_{1}\epsilon.
\end{gather*}
This inequality holds for an arbitrary $\epsilon>0$, so we obtain \eqref{eq:ineqtobeproved}. This finishes the proof of \eqref{eq:secorderasympRieszenergy:1}.

The asymptotic formula \eqref{secondorderasympRieszenergy:2} follows from the inequality $E_{s}(\alpha_{N})\geq \mathcal{L}_{s}(N)$, which is valid for every $N$ and is an equality for all $N$ of the form $N=2^{n}$, and the asymptotic formula \eqref{eq:secorderasympRieszFekete}.\hfill$\Box$

\subsection{First-order asymptotics in the Riesz case for $s>1$}

\noindent\emph{Proof of Theorem~$\ref{theo:firstorderasympRieszenergy}$.} The proofs of \eqref{eq:firstorderasympRieszenergy:1} and \eqref{eq:firstorderasympRieszenergy:2} are identical to the proofs of the corresponding formulas in Theorem~\ref{theo:secorderasympRieszenergy:1}. The reader only needs to use, instead of  \eqref{eq:secorderRieszenergy}, the formula
\begin{align*}
\frac{E_{s}(\alpha_{N})}{N^{1+s}}= & 
\sum_{k=1}^{t-1}\left(\frac{2^{n_{k}+1}}{N}\right)^{1+s}\left(\sum_{j=k+1}^{t}2^{n_{j}-n_{k}}\right)\frac{\mathcal{L}_{s}(2^{n_{k}+1})}{(2^{n_{k}+1})^{1+s}}\\ 
+ & \sum_{k=1}^{t}\left(\frac{2^{n_{k}}}{N}\right)^{1+s}\left(1-\sum_{j=k+1}^{t}2^{n_{j}-n_{k}+1}\right)\frac{\mathcal{L}_{s}(2^{n_{k}})}{(2^{n_{k}})^{1+s}},
\end{align*}
which follows from \eqref{Rieszenergybinary}, and use \eqref{eq:asympfirstorderRieszsupercrit} instead of \eqref{eq:secorderasympRieszFekete}.\hfill$\Box$

\bigskip

Figure~\ref{fig:energysonehalf} below displays the first $4200$ points of the sequence $\left(E_{s}(\alpha_{N})/N^{1+s}\right)_{N}$ in the case $s=2$.

\begin{figure}[h]
\centering
\includegraphics[totalheight=3in,keepaspectratio]{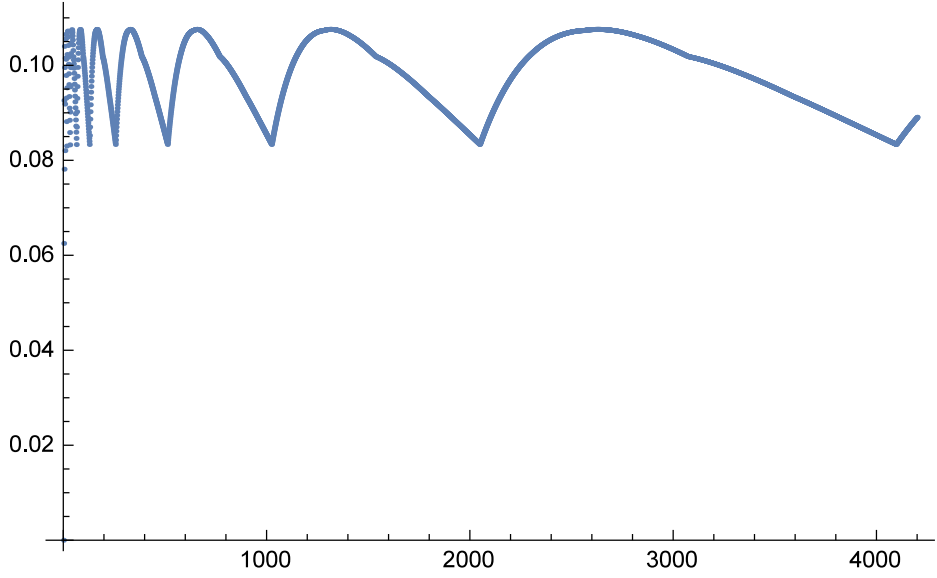}
\caption{The first $4200$ points of the sequence $\left(E_{s}(\alpha_{N})/N^{1+s}\right)_{N}$ in the case $s=2$.}
\label{fig:energysonehalf}
\end{figure}

\subsection{Second-order asymptotics in the Riesz case for $s=1$}

\noindent\emph{Proof of Theorem~$\ref{theo:criticalcase}$.} Below we will use the notation
\begin{equation}\label{def:R1}
\mathcal{R}_{1}(N):=\frac{\mathcal{L}_{1}(N)-\frac{1}{\pi}N^{2}\log N}{N^2}.
\end{equation}
If $N=2^{n_{1}}+2^{n_{2}}+\cdots+2^{n_{p}}$ in decreasing order of powers, applying \eqref{expansionN2} we can write conveniently
\begin{align*}
N^{2} \log N & =\sum_{k=1}^{p-1}\left(\sum_{j=k+1}^{p}2^{n_{j}-n_{k}}\right) \left(2^{2(n_{k}+1)} \log\left(2^{n_{k}+1}\right)+2^{2(n_{k}+1)} \log\left(\frac{N}{2^{n_{k}+1}}\right)\right)\\
 & +\sum_{k=1}^{p}\left(1-\sum_{j=k+1}^{p} 2^{n_{j}-n_{k}+1}\right)\left(2^{2n_{k}}\log \left(2^{2n_{k}}\right)+2^{2n_{k}} \log\left(\frac{N}{2^{2n_{k}}}\right)\right).
\end{align*}
Hence, applying \eqref{Rieszenergybinary} for $s=1$ and \eqref{def:R1} we obtain
\begin{align}\label{eq:EnergyNcritcase}
\frac{E_{1}(\alpha_{N})-\frac{1}{\pi}N^2\log N}{N^{2}} & =\frac{1}{N^2}\sum_{k=1}^{p-1}\left(\sum_{j=k+1}^{p}2^{n_{j}+n_{k}+2}\right)
\left(\mathcal{R}_{1}\left(2^{n_{k}+1}\right)+\frac{1}{\pi}\log\left(\frac{2^{n_{k}+1}}{N}\right)\right)\notag\\
& +\frac{1}{N^2}\sum_{k=1}^{p}\left(2^{2n_{k}}-\sum_{j=k+1}^{p}2^{n_{j}+n_{k}+1}\right)\left(\mathcal{R}_{1}(2^{n_{k}})+\frac{1}{\pi}\log\left(\frac{2^{n_{k}}}{N}\right)\right).
\end{align}

The proof of \eqref{secondorderasympcritcase:1} follows the same guidelines of the proof of \eqref{eq:secorderasympRieszenergy:1}. To prove the inequality ``$\geq$" in \eqref{secondorderasympcritcase:1}, we take an arbitrary $\vec{\theta}=(\theta_{1},\ldots,\theta_{p})\in\Theta_{p}$, and we let $\mathcal{N}$ be a sequence of integers $N=2^{n_{1}}+\cdots+2^{n_{p}}$ as in Definition~\ref{def:Thetap} satisfying \eqref{eq:conddefthetavec}. If we call $L=\frac{1}{\pi}(\gamma-\log (\pi/2))$ and apply \eqref{eq:conddefthetavec} and \eqref{eq:secorderFeketecrit}, it follows from \eqref{eq:EnergyNcritcase} that
\begin{gather}
\lim_{N\in\mathcal{N}}\frac{E_{1}(\alpha_{N})-\frac{1}{\pi}N^{2} \log N}{N^2}\label{limitfinitetau}\\
=\sum_{k=1}^{p-1}\left(\sum_{j=k+1}^{p}4 \theta_{k} \theta_{j}\right)(L+\frac{1}{\pi}\log(2\theta_{k}))+\sum_{k=1}^{p}(\theta_{k}^{2}-\sum_{j=k+1}^{p}2 \theta_{k} \theta_{j})(L+\frac{1}{\pi}\log \theta_{k})\notag\\
=L\left(\sum_{k=1}^{p}\theta_{k}^{2}+2\sum_{k=1}^{p-1}\sum_{j=k+1}^{p}\theta_{k} \theta_{j}\right)
+\frac{4}{\pi}\sum_{k=1}^{p-1}\left(\sum_{j=k+1}^{p}\theta_{j}\right)\theta_{k} \log(2\theta_{k})
+\frac{1}{\pi}\sum_{k=1}^{p}(\theta_{k}^2-\sum_{j=k+1}^{p}2\theta_{k}\theta_{j})\log \theta_{k}\notag\\
=L+\frac{1}{\pi}\,K(\theta_{1},\ldots,\theta_{p})\leq L+\frac{\kappa}{\pi},\notag
\end{gather}
where we used the fact that 
\[
\sum_{k=1}^{p}\theta_{k}^{2}+2\sum_{k=1}^{p-1}\sum_{j=k+1}^{p}\theta_{k} \theta_{j}=(\theta_{1}+\cdots+\theta_{p})^2=1.
\]
This proves the desired inequality.

The proof of the converse inequality in \eqref{secondorderasympcritcase:1} is similar to the one given for \eqref{eq:ineqtobeproved}, so we will make reference to that proof below. Let $\mathcal{N}\subset\mathbb{N}$ be an infinite sequence for which $\left(\frac{E_{1}(\alpha_{N})-\frac{1}{\pi}N^2 \log N}{N^2}\right)_{N\in\mathcal{N}}$ converges and we shall show that
\begin{equation}\label{eq:ineqtobeproved2}
\lim_{N\in\mathcal{N}}\frac{E_{1}(\alpha_{N})-\frac{1}{\pi}N^2 \log N}{N^2}\leq L+\frac{\kappa}{\pi}.
\end{equation}
As in the proof of \eqref{eq:ineqtobeproved}, if there exists $p\geq 1$ such that an infinite number of $N\in\mathcal{N}$ satisfy $\tau(N)=p$, then it is clear that \eqref{eq:ineqtobeproved2} holds.

So we assume now that $\tau(N)\rightarrow\infty$ as $N\rightarrow\infty$ in the sequence $\mathcal{N}$. We have, for $N=\sum_{k=1}^{\tau(N)} 2^{n_{k}}$, 
\begin{align}
\frac{E_{1}(\alpha_{N})-\frac{1}{\pi}N^2 \log N}{N^2}= & \sum_{k=1}^{\tau(N)}\left(\frac{2^{n_{k}}}{N}\right)^{2}\left(\mathcal{R}_{1}(2^{n_{k}})+2^{-n_{k}+1}\,(2\,\mathcal{R}_{1}(2^{n_{k}+1})-\mathcal{R}_{1}(2^{n_{k}})) \sum_{j=k+1}^{\tau(N)}2^{n_{j}}\right)\notag\\
+ & \sum_{k=1}^{\tau(N)}\left(\frac{2^{n_{k}}}{N}\right)^{2}\left(r(2^{n_{k}})+2^{-n_{k}+1}\,(2\,r(2^{n_{k}+1})-r(2^{n_{k}})) \sum_{j=k+1}^{\tau(N)}2^{n_{j}}\right),\label{eq:decompE1norm}
\end{align}
where we use the notation
\[
r(2^{n_{k}})=\frac{1}{\pi} \log(2^{n_{k}}/N),\qquad r(2^{n_{k}+1})=\frac{1}{\pi} \log(2^{n_{k}+1}/N).
\]
Let $\epsilon>0$ be arbitrary, and choose $M\in\mathbb{N}$ sufficiently large so that \eqref{bound4} holds. Let $\lambda_{N,k}$ denote the expression in \eqref{def:lambdaNk:2} with $s=1$, and let
\[
\rho_{N,k}:=r(2^{n_{k}})+2^{-n_{k}+1}\,(2\,r(2^{n_{k}+1})-r(2^{n_{k}})) \sum_{j=k+1}^{\tau(N)}2^{n_{j}}.
\]
We see from \eqref{eq:decompE1norm} that we can write
\begin{equation}\label{eq:decompE1norm:2}
\frac{E_{1}(\alpha_{N})-\frac{1}{\pi}N^2 \log N}{N^2}=S_{N,M,1}+S_{N,M,2}+S_{N,M,3}+S_{N,M,4},
\end{equation}
where $S_{N,M,1}$ and $S_{N,M,2}$ are defined in \eqref{def:SNM1SNM2} (taking $s=1$), and 
\[
S_{N,M,3}:=\sum_{k=1}^{M}\left(\frac{2^{n_{k}}}{N}\right)^2 \rho_{N,k},\qquad S_{N,M,4}:=\sum_{k=M+1}^{\tau(N)}\left(\frac{2^{n_{k}}}{N}\right)^2 \rho_{N,k}.
\]
As in \eqref{bound1} we have
\[
|\lambda_{N,k}|\leq C_{1},\qquad \mbox{for all}\,\,N\in\mathcal{N}\,\,\mbox{and}\,\,k=1,\ldots,\tau(N),
\]
for some constant $C_{1}>0$. Therefore as in \eqref{bound5} we have
\[
|S_{N,M,2}|< C_{1} \epsilon,\qquad \mbox{for all}\,\,N\in\mathcal{N}.
\]
We again write
\[
S_{N,M,1}=D_{N,M,1}+D_{N,M,2}
\]
with $D_{N,M,1}$ and $D_{N,M,2}$ given by \eqref{def:DNM1} and \eqref{def:DNM2}, respectively, taking $s=1$ in these formulas. We also have the estimate \eqref{bound6}. If we use \eqref{energydecomp:3}, the bound \eqref{bound7} and the previous estimates, we conclude that
\begin{gather}
S_{N,M,1}+S_{N,M,2}\label{estimateSNM1SNM2}\\
=\sum_{k=1}^{M}\left(\frac{2^{n_{k}}}{2^{n_{1}}+\cdots+2^{n_{M}}}\right)^2\left(\mathcal{R}_{1}(2^{n_{k}})+2^{-n_{k}+1}\,(2\,\mathcal{R}_{1}(2^{n_{k}+1})-\mathcal{R}_{1}(2^{n_{k}})) \sum_{j=k+1}^{M}2^{n_{j}}\right)+O(\epsilon).\notag
\end{gather}

The analysis for the sum $S_{N,M,3}+S_{N,M,4}$ follows the same argument, so we will not reproduce it below. Now we need to take into account the following estimates, which are easy to check: There exists an absolute constant $C>0$, independent of $N$ and $M$, such that
\begin{align*}
\left|\frac{2^{n_{k}}}{N}\,\rho_{N,k}\right| & <C,\qquad \mbox{for all}\,\,N\in\mathcal{N}\,\,\mbox{and}\,\,k=1,\ldots,\tau(N), \\
\sum_{k=1}^{M}\frac{2^{n_{k}}}{N}\,\left|\log\left(\frac{2^{n_{k}}}{N}\right)\right| & <C,\qquad \mbox{for all}\,\,N\in\mathcal{N}\,\,\mbox{and}\,\,M< \tau(N). 
\end{align*}
Using these estimates we find similarly that
\begin{gather}
S_{N,M,3}+S_{N,M,4}\label{estimateSNM3SNM4}\\
=\sum_{k=1}^{M}\left(\frac{2^{n_{k}}}{2^{n_{1}}+\cdots+2^{n_{M}}}\right)^2\left(\widetilde{r}(2^{n_{k}})+2^{-n_{k}+1}\,(2\,\widetilde{r}(2^{n_{k}+1})-\widetilde{r}(2^{n_{k}})) \sum_{j=k+1}^{M}2^{n_{j}}\right)+O(\epsilon),\notag
\end{gather}
where we use the notation
\[
\widetilde{r}(2^{n_{k}})=\frac{1}{\pi}\log\left(\frac{2^{n_{k}}}{2^{n_{1}}+\cdots+2^{n_{M}}}\right),\qquad \widetilde{r}(2^{n_{k}+1})=\frac{1}{\pi}\log\left(\frac{2^{n_{k}+1}}{2^{n_{1}}+\cdots+2^{n_{M}}}\right).
\]

Finally, we let $\widetilde{\mathcal{N}}$ be a subsequence of $\mathcal{N}$ such that the limits \eqref{cond:subseqNtilde} hold. Then, as in \eqref{limitfinitetau} we see that along the subsequence $\widetilde{\mathcal{N}}$, the first expression on the right-hand side of \eqref{estimateSNM1SNM2} converges to $L$, and the first expression on the right-hand side of \eqref{estimateSNM3SNM4} converges to $(1/\pi)\,K(\theta_{1},\ldots,\theta_{M})$. Therefore, applying \eqref{eq:decompE1norm:2}, \eqref{estimateSNM1SNM2} and \eqref{estimateSNM3SNM4}, we conclude that
\begin{gather*}
\lim_{N\in\mathcal{N}}\frac{E_{1}(\alpha_{N})-\frac{1}{\pi}N^{2}\log N}{N^{2}}
=\lim_{N\in\widetilde{\mathcal{N}}}\frac{E_{1}(\alpha_{N})-\frac{1}{\pi}N^{2}\log N}{N^{2}}\\=\lim_{N\in\widetilde{\mathcal{N}}}\left(S_{N,M,1}+S_{N,M,2}+S_{N,M,3}+S_{N,M,4}\right)
\leq L+\frac{1}{\pi} K(\theta_{1},\ldots,\theta_{M})+O(\epsilon)\leq L+\frac{\kappa}{\pi}+O(\epsilon).
\end{gather*}
This proves \eqref{eq:ineqtobeproved2} since $\epsilon$ is arbitrary.

The formula \eqref{secondorderasympcritcase:2} follows immediately from \eqref{eq:secorderFeketecrit} and the inequality $E_{1}(\alpha_{N})\geq \mathcal{L}_{1}(N)$, which is an equality for all $N$ of the form $N=2^{n}$. \hfill$\Box$

\bigskip

The following figure shows the first $4200$ values of the sequence $\left(\frac{E_{1}(\alpha_{N})-\frac{1}{\pi}N^{2}\log N}{N^2}\right)_{N}$.

\begin{figure}[h!]
\centering
\includegraphics[totalheight=3in,keepaspectratio]{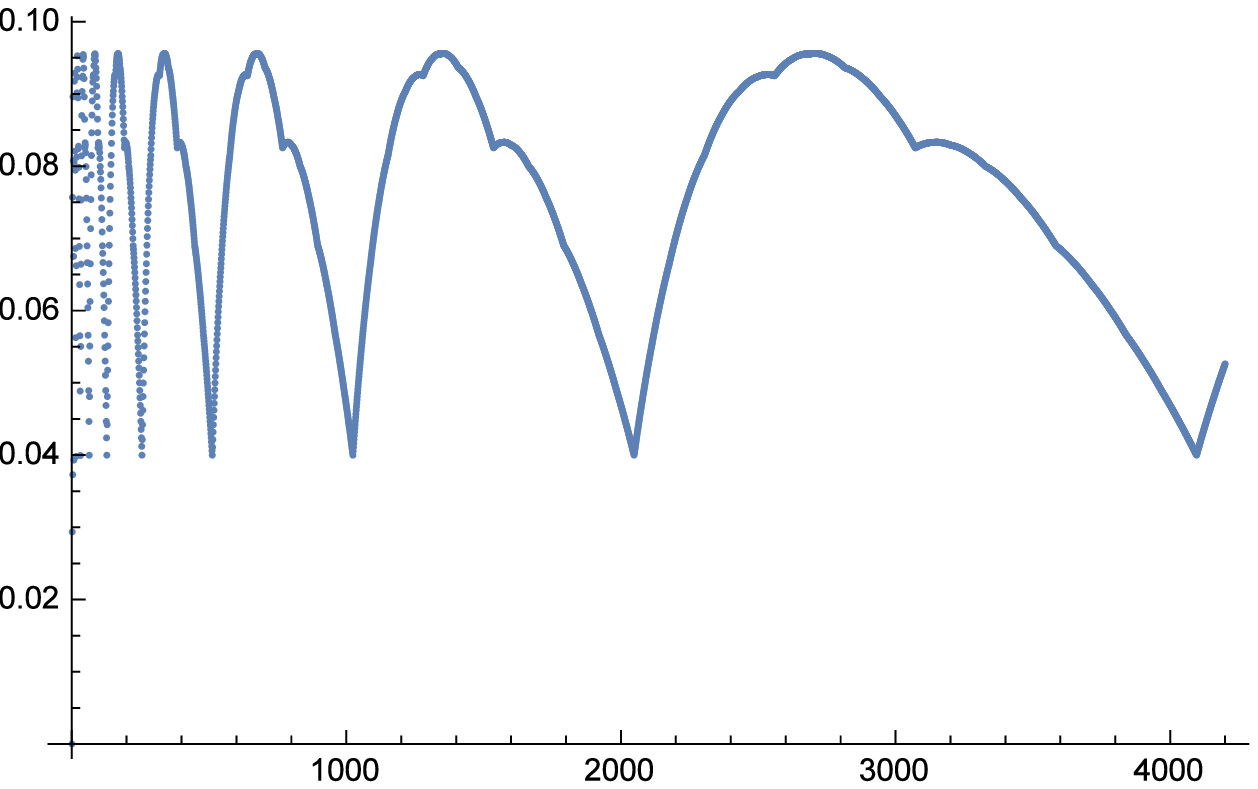}
\label{fig:sequences1first4200}
\end{figure}

\bigskip

\noindent\textbf{Acknowledgments}

\smallskip

\noindent We thank the referee for the many valuable comments provided.

\end{document}